\documentclass[12pt, reqno]{amsart}
	
\usepackage[utf8]{inputenc}
\usepackage[T1]{fontenc}
\usepackage{tikz,tikz-3dplot}
\usepackage{amsmath,amsfonts,amssymb,amsmath,latexsym,mathrsfs}
\usepackage{shuffle}
\usepackage[all]{xy}
\usepackage{stackengine}

\usepackage{tabu}
\usepackage{tikz-cd} 

\usepackage{comment}
\usepackage{url}
\usepackage{tikz}
\newtheorem{theorem}{Theorem}[section]

\newtheorem{proposition}[theorem]{Proposition}
\newtheorem{notation}[theorem]{Notation}
\newtheorem{definition}[theorem]{Definition}

\newtheorem{corollary}[theorem]{Corollary}
\newtheorem{prop-def}[theorem]{Définition-Proposition}
\newtheorem{remark}[theorem]{Remark}

\newtheorem{lemma}[theorem]{Lemma}
\newcommand{\sh}{\shuffle}
\newcommand{\Q}{ {\mathbb Q} }
\newcommand{\R}{ {\mathbb R} }

\newcommand{\Z}{ {\mathbb Z} }
\newcommand{\CZ}{ {\mathcal{Z}} }
\newcommand{\CM}{ {\mathcal{M}} }
\newcommand{\HT}{ {\hat{\otimes}} }

\newcommand{\Wsh}{ {W_\sh }}
\newcommand{\Wst}{ {W_\star }}

\newcommand{\N}{ {\mathbb N} }

\newcommand{\K}{ {\Bbbk} }
\newcommand{\uk}{ {\underline{k}} }

\newcommand{\D}{ {D_{\bullet \bullet}} }
\newcommand{\DG}{ {\mathscr{D}_{\bullet \bullet}(G)} }
\newcommand{\DGH}{ {\widehat{\mathscr{D}}_{\bullet \bullet}(G)} }
\newcommand{\DGHE}{ {\widehat{\mathscr{D}}_{\bullet \bullet}(\{e\})} }

\newcommand{\ls}{ \mathfrak{ls}}
\newcommand{\dmr}{ \mathfrak{dmr}_0}
\newcommand{\ds}{ \mathfrak{ds}}
\newcommand{\Dsh}{ \mathrm{Dsh}}
\title{Bigraded Lie algebras related to MZVs}
\author{Mohamad MAASSARANI}
\address{IRMA, Université de Strasbourg, 7 rue René Descartes, 67084
Strasbourg, France}
\email{maassarani@math.unistra.fr}

\begin{document}
\maketitle
\begin{abstract}
We prove that the dihedral Lie coalgebra $\D:={\oplus}_{k\geq m \geq 1} D_{m,k}$ corresponding to $\DGH$ of \cite{Gon} for $G=\{e\}$  is the bigraded dual of the linearized double shuffle Lie algebra $\ls:={\oplus}_{k\geq m \geq 1}\ls_m^k\subset \Q\langle x,z \rangle$ of \cite{Brown} whose Lie bracket is the Ihara bracket initially defined over $\Q\langle x,z \rangle$. This by constructing an explicit isomorphism of bigraded Lie coalgebras $\D \to \ls^\vee$, where $\ls^\vee$ is the Lie coalgebra dual (in the bigraded sense) to $\ls$. The work leads to the equivalence between the two statements: "$\D$ is a Lie coalgebra with respect to Goncharov's cobraket formula in \cite{Gon}" and "$\ls$ is preserved by the Ihara bracket". We also prove folklore results from \cite{Brown} and \cite{IKZ} (that apparently have no written proofs in the literature) stating that for $m \geq 2$: $D_{m,\bullet}:=\oplus_{k\geq m} D_{m,k}$ is graded isomorphic (dual) to the double shuffle space  $\Dsh_m:=\oplus_{k\geq m} \mathrm{Dsh}_{m}({{k}-m})  \subset \Q[x_1,\dots,x_m]$ of \cite{IKZ} (stated in \cite{IKZ}), and that the linear map $f_m: \Q\langle x,z \rangle_m \to \Q[x_1,\dots,x_m]$, where $\Q\langle x,z \rangle_m$ is the space linearly generated by monomials of $\Q\langle x,z \rangle$ of degree $m$ with respect to $z$, given by $x^{n_1}z\cdots x^{n_m}zx^{n_{m+1}+1}\mapsto \delta_{0,n_{m+1}} x_1^{n_1}\cdots x_{n_m}^{n_m}$, with $\delta_{a,b}$ the Kronecker delta, restricts to a graded isomorphism $\bar{f}_m:\ls_m:=\oplus_{k\geq m} \ls_m^k \to \mathrm{Dsh}_{m}$ (stated in \cite{Brown}). Here, we establish three explicit compatible isomorphisms $D_{\bullet \bullet} \to  \mathfrak{ls}^\vee, D_{m\bullet}\to \mathrm{Dsh}_{m}^\vee$ and $\bar{f}_m: \mathfrak{ls}_m \to \mathrm{Dsh}_{m}$, where $\mathrm{Dsh}_{m}^\vee$ is the graded dual of $\mathrm{Dsh}_{m}$.
\end{abstract}

\section*{Introduction}

\subsection*{Context and main results }\quad \\\\
Multiple zeta values (MZVs) are real numbers generalising the values of the Riemann zeta function at positive integers. These numbers were first studied by Euler then reappeared recently in geometry, knot theory, quantum algebra and arithmetic geometry. The study of relations between MZVs over $\Q$ is a well known problem subject to many conjectures.\\\\
A MZV is a real number of the form:
\begin{equation}\label{MZV} \zeta(\uk)= \sum_{n_1>\dots>n_m \geq 1} \frac{1}{n_1^{k_1}\cdots n_m^{k_m}}, \end{equation}
where $\uk=(k_1,\cdots,k_m) \in (\N^*)^m$ with $k_1 \geq 2$ and $m\geq 1$. We say that $\zeta(\uk)$ is of depth $m$ and weight $k=k_1+\cdots+k_m$. In the literature, we find various polynomial relations between MZVs over $\Q$. For instance double shuffle relations (\cite{IKZ}) are $\Q$-linear relations between MZVs. One can also obtain relations between MZVs by comparing different regularisations of divergent MZVs ($\zeta(\uk)$ with $k_1=1$). These new relations with the double shuffle relations are called the regularised (or also known as extended) double shuffle relations (\cite{IKZ}). One of the main conjectures on MZVs is that the extended double shuffle relations suffice to describe the $\Q$-subalgebra $\CZ$ of $\R$ generated by the MZVs (\cite{IKZ}). We describe briefly the conjecture. Denote by $\tilde{\CZ}^f$ the $\Q$-algebra linearly spanned by $1$ and the symbols $\zeta^f$ (for $\zeta$ running over all the MZVs) equipped with a product mimicking the shuffle product (\cite{IKZ}) of  MZVs. Now consider the quotient $\CZ^f$ of $\tilde{\CZ}^f$ by the formal analogues of the extended double shuffle relations. The main conjecture is equivalent to the following statement:
\begin{center}
the mapping $\CZ^f \to \CZ$ given by $\zeta^f \mapsto \zeta$ is an isomorphism of algebras. 
\end{center} 
The algebra $\CZ^f$ described here is the algebra $R_{EDS}$ of \cite{IKZ} and the above statement corresponds to conjecture 1 of the same paper.\\\\
Denote by $\CZ_{{+}}$ the vector subspace of $\CZ$ generated by the MZVs and denote by $ \CZ_k^{(m)}$ (for $k \geq m \geq 1$) the subspace of $\CZ$ generated by MZV of weight $k$ and depth less than or equal to $m$ (we also set $\CZ_k^{(0)}=0$). The standard conjectures lead to the study of: 
$$  \CM:= \underset{k\geq m \geq 1}{\oplus} \CM_{k,m},\qquad    \CM_{k,m}:= \CZ_k^{(m)}/  (\CZ_k^{(m-1)}+ \CZ_k^{(m)} \cap \CZ_{{+}}^2).  $$
Conjecturally, the space $\CM$ should be isomorphic to the associated graded of $\CZ_{{+}}/ \CZ_{{+}}^2$ with respect to the weight filtration and the dimension $C_{k,m}$ of $\CM_{k,m}$ should correspond to the number of generators of $\CZ$ of depth $m$ and weight $k$. A conjectural formula for $C_{k,m}$ is known (appendix of \cite{IKZ}).\\\\
For $G$ an abelian group, Goncharov introduces in \cite{Gon1} a bigraded Lie coalgebra called the dihedral Lie coalgebra $\D(G)$  also denoted by $\DGH$ in \cite{Gon}. In this paper, we mean by the dihedral Lie coalgerbra the coalgebra $\DGH$ for $G=\{e\}$ which we denote by $\D$. The bigraded vector space $\CM$ is a quotient vector space of $\D$. Goncharov computes the dimension of the bidegree $(m,k)$ part $D_{m,k}$ (of weight $k$ and depth $m$) of $\D$, for $m=1,2,3$. He also proves that $D_{m,k}=0$  if $k+m$ is odd (parity result). From theses results he deduces that as conjectured $C_{k,m}=0$ if $k+m$ is odd and gives upper bounds for $C_{k,1}, C_{k,2},C_{k,3}$ corresponding to the conjectural values of these numbers. The cohomology of $\D$ plays an important role in the computations, this explains the importance of the Lie structure of $\D$.\\\\

In (\cite{Racinet}) Racinet introduces a prounipotent group scheme $\mathrm{DMR}_0$ whose Lie algebra is $\dmr$ which Lie bracket is the Ihara bracket. The Lie algebra $\dmr$ is a complete graded (for weight) Lie algebra equipped with a depth decreasing filtration compatible to the "grading". One can consider a Lie subalgebra $\ds$ of $\dmr$ graded for weight and whose degree completion (for weight) is isomorphic to $\dmr$. The coordinate ring $\mathcal{O}( \mathrm{DMR}_0)$ of $\mathrm{DMR}_0$ and $\bar{\CZ}^f:= \CZ^f/ \zeta(2)^f \CZ^f$ are isomorphic algebras and therefore the graded dual $\mathcal{U}( \ds)^\vee$ (for weight) of the enveloping algebra $\mathcal{U}( \ds)$ of $\ds$ is isomorphic to the algebra $\bar{\CZ}^f$ (\cite{EL}). This implies that the graded dual $\ds^\vee$ (for weight) of $\ds$ is isomorphic to $\bar{\CZ}_{{+}}^f/(\bar{\CZ}_{{+}}^f)^2$, where $\bar{\CZ}_{{+}}^f \subset \bar{\CZ}^f$ is the ideal generated by the symbols $\zeta^f$.\\\\
A bigraded version $\ls$ (linearized double shuffle space) of the vector space $\dmr$ (or $\ds$) is defined by Brown (\cite{Brown}). The space $\ls$ is a bigraded subspace of the free associative algebra $\Q\langle x,z \rangle$ on the indeterminates $x,z$. Here $\Q\langle x,z \rangle$ is bigraded with respect to the total degree (weight) and the partial degree with respect to $z$ (depth). One has a natural injective map from the associated graded of $\dmr$ (and $\ds$) with respect to depth filtration to $\ls$. This map should be an isomorphism except in bidegree $(1,1)$ (\cite{Brown}). In the same paper Brown states that one can show that $\ls$ is preserved by the Ihara bracket $\{-,-\}$ (defined over $\Q\langle x,z \rangle$) by adapting the work of \cite{Racinet} for $\dmr$.
Schneps proves in \cite{Schnepsari} (see theorem 3.4.3 and its proof) that $\ls$ is preserved by the Ihara bracket. The proof uses the theory of (bi)moulds introduced by \'Ecalle (see for example \cite{Ecalle}) and an analogy between the theory of (bi)moulds and series in non-commutative variables $x,z$ established by Racinet in \cite{RacinetPhd}.\\\\
In this paper we show (subsection \ref{proof a,b}) that: 
\begin{itemize}
\item[(a)] the Lie coalgebra $\D$ and the Lie algebra $\ls$ equipped with the Ihara bracket $\{-,-\}$ are dual in the bigraded sense.
\end{itemize}
For that, we construct an explicit isomorphism of bigraded Lie coalgebras $\D \to \ls^\vee$, where $\ls^\vee$ is the bigraded dual of $\ls$ equipped with cobracket dual to $\{-,-\}_{\vert \ls}$.\\\\
The work leads to the following equivalence (subsection \ref{proof a,b}) between the results of Goncharov and Brown-Schneps:
\begin{itemize}
\item[(b)] The bigraded vector space $\D$ is a Lie coalgebra under the formulas given by Goncharov (\cite{Gon}) if and only if the bigraded space $\ls$ is preserved by the Ihara bracket.\\
\end{itemize}
In \cite{IKZ}, the authors construct, for $m\geq 2$, the double shuffle subspace:
$$\Dsh_m={\oplus}_{k \geq m}\mathrm{Dsh}_{m}({{k}-m}) \subset \Q[x_1,\dots,x_m].$$ The space $\CM_{k,m}$ is naturally a quotient  of the dual of the double shuffle space $\mathrm{Dsh}_m(k-m)$. It is conjectured in \cite{IKZ} that $\mathrm{Dsh}_m(k-m)$ and $\CM_{k,m}$ have the same dimensions for $k\geq m \geq 2$. The authors also compute the dimension of $\mathrm{Dsh}_m(d)$ for $m=1,2$ (give estimates for $m=3$) and they obtain a parity result similar to the one obtained for $\D$ in \cite{Gon}. These results give the same upper bounds for $C_{k,m}$ (for $m=1,2$) obtained by Goncharov and also imply the parity result for $C_{k,m}$.\\\\ 
In \cite{IKZ}, footnote page 335, it is stated (without a proof) that the spaces $D_{m,k}$ and $\mathrm{Dsh}_m(k-m)$ are isomorphic. Therefore $D_{m,k}$ and $\CM_{k,m}$ are conjecturally isomorphic for $k\geq m \geq 2$. In subsection \ref{Subsection Dsh ls D} (paragraph \ref{Proof (c)}), we prove that:
\begin{itemize}
\item[(c)] For $m\geq 2$, the depth $m$ part $D_{m,\bullet}$ of $\D$ is isomorphic as a graded space (for weight) to the graded dual $\Dsh_m^\vee$ of $\Dsh_m=\oplus_{k\geq m}\mathrm{Dsh}_{m}({{k}-m})$. 
\end{itemize}
To do so we contruct an explicit isomorphism of graded spaces $D_{m,\bullet} \to \Dsh_m^\vee$.\\\\
 In \cite{Brown}, Brown also states, without giving a proof that the map $f_m :\Q\langle x,z \rangle_m \to \Q[x_1,\cdots,x_m]$ (defined in $(d)$ below) restricts to an isomorphism of graded spaces between the depth $m$ part $\ls_m$ of $\ls$ and $\Dsh_m$  (for $m \geq 2$). In subsection \ref{Subsection Dsh ls D} (paragraph \ref{Proof (d)}), we prove that:
\begin{itemize}
\item[(d)] For $m\geq2$, the linear map $f_m :\Q\langle x,z \rangle_m \to \Q[x_1,\cdots,x_m], x^{n_1}z\cdots x^{n_m}zx^{a} \mapsto \delta_{a0} x_1^{n_1} \cdots x_m^{n_m}$, where $\delta_{a0}$ is the Kronecker delta and $\Q\langle x,z \rangle_m$ is the depth $m$ part of $\Q\langle x,z \rangle$, restricts to an isomorphism $\ls_m \to \Dsh_m$ of weight graded spaces compatible to the isomorphisms constructed to show $(a)$ and $(c)$.\\ 
\end{itemize}
By combining the isomorphisms constructed to prove $(a)$ and $(c)$ we get an isomorphism of graded spaces between $\ls_m $ and $\Dsh_m$. Result $(d)$ allows to understand this isomorphism. Note that the duality between \textbf{the underlying bigraded vector spaces} of $\D$ and $\ls$ (in depth $m\geq 2$) can be deduced from $(c)$ and $(d)$.\\\\
The following table, gives a clearer view of the results in this paper and the contributions of other authors:
\[\begin{tabu} to 0.8\textwidth { | X[c] | X[c] | X[c] | }
 \hline
Result & Stated in & Proved in \\
 \hline
(a)   &  This paper &  This paper  \\
\hline
 (b)  & This paper & This paper \\
\hline
(c)  & \cite{IKZ}  & This paper  \\
\hline
(d)  & \cite{Brown}  & This paper \\
\hline

\end{tabu}\]
\subsection*{Outline of the paper}\quad \\\\
Through the paper $\K$ is a field of characteristic zero. The spaces $\D, \ls$ and $\Dsh_m$ are originally $\Q$-vector spaces and we should take $\K=\Q$, but the definitions of these spaces and the results can be easily given for any field of characteristic zero.\\\\
The first section contains reminders that will be used in the sequel of the pape on: free associative algebras, pairings of bigraded spaces, adjoints, duality between bigraded Lie algebras and Lie coalgebras, and duality between some spaces related by given series.\\\\
In the second section we recall the definition of the linearized double shuffle Lie algebra $\ls$ (\cite{Brown}). In subsection \ref{ssection ls} we introduce the bigraded space $\ls \subset \K\langle x,z \rangle$ and define the Ihara Lie bracket $\{-,-\} : \K\langle x,z \rangle^{\otimes 2} \to \K\langle x,z \rangle$ preserving $\ls$. We then give in subsection \ref{def Dshm} the definition of the double shuffle space $\Dsh_m$ (for $m\geq 2$) of \cite{IKZ}.\\\\
Section \ref{diheral coalgebra} is devoted to the dihedral Lie colagebra $\D$. An alternative definition of $\D$ more suitable for the paper is given. We define, using generating series similar to those in \cite{Gon}, a bigraded quotient space $W/W_R$, where $W$ is a free vector space and construct in proposition \ref{definition dihedral} (and it's proof) a natural isomorphism of bigraded spaces $\eta: \D \to W/W_R$. Then, we define a map $\tilde{\delta}: W \to W\otimes W$ inducing a bigraded Lie cobracket $\delta$ on $W/W_R $ for which $\eta:\D \to W/W_R$ becomes a Lie isomorphism (theorem \ref{Lie structure dihedral}). The construction of $W/W_R$ and $\delta$ is a transcription of the work of \cite{Gon}. We also define extensions $V$ and $F$ of $W$ and $W_R$ by a bigraded free vector space $U$, to get a new quotient space $V/F\simeq W/W_R$ that will be used in the next sections.\\\\
In section \ref{Section series}, we introduce a family of series $\{Q_m\}_m$ in commuting variables with values in the free associative algebra $\K\langle x,z \rangle$ on two indeterminates $x,z$, and show that the shuffle product (a product defined over $\K\langle x,z \rangle$ denoted by $\sh$)  of $Q_p$ and $Q_q$ is given by: 
$$ Q_p(t_1,\dots,t_p) \sh Q_q(t_{p+1},\dots,t_{p+q})=\underset{\sigma \in S(p,q)}{\sum} Q_{p+q}(t_{\sigma^{-1}(1)},\dots, t_{\sigma^{-1}(p+q)}),$$
where $S(p,q)$ is the set of $(p,q)$-shuffles.  This formula is a formal analogue of formula (15) of \cite{Gon1} for generating series of iterated integrals in the case where the iterated integrals represent MZVs and it will be used in section \ref{section duality}.\\\\
In subsection \ref{orthogonal ls}, we give a decomposition (corollary \ref{ls orth}) of the orthogonal complement $\ls^{\perp}$ of $\ls$ with respect to the "canonical pairing" $\langle - , -\rangle$ of $\K \langle x,z \rangle$, the pairing for which unitary monomials form an orthonormal basis. In subsection \ref{morphism phi}, we introduce in proposition \ref{morphisme phi} an isomorphism of bigraded spaces $\phi : V\to \K \langle x,z \rangle$ and then show (corollary \ref{phi F}) that the image of $F$ by $\phi$ is $\ls^{\perp}$. From these results we deduce in subsection \ref{proof orthogonality} that $F\subset V$ is the space orthogonal to $\ls$ with respect to a perfect pairing $\langle\phi(-),-\rangle: V\otimes \K \langle x,z \rangle \to \K$ denoted by $\langle-,-\rangle_{\phi}$.\\\\
In  subsection \ref{Ih Co}, we compute the adjoint $\{-,-\}^*$ of $\{-,-\}$ with respect to the "canonical pairings" $\langle - , -\rangle$ and $\langle - , -\rangle^{\otimes2}$ of $\K \langle x,z \rangle$ and $\K \langle x,z \rangle^{\otimes 2}$. More precisely, we compute (proposition \ref{coIh P}), for $m\geq 0$, the image of a generating series $P(v_0,\dots,v_m)$ of unitary monomials of depth $m$ of $\K \langle x,z \rangle$ by $\{-,-\}^*$.  For that we introduce different operators allowing us to decompose $\{-,-\}^*$ (definition \ref{def di} and proposition \ref{decomp coIhara}), we describe these operators (proposition \ref{di*}) to then compute (lemma \ref{lemma dP}) the image of $P(v_0,\dots,v_m)$ by these operators. In subsection \ref{sec comp Ih delta}, we deduce (proposition \ref{comp Ih delta}) using the formula for $\{-,-\}^*(P(v_0,\dots,v_m))$ that $(\tilde{\delta} -\phi^\star \{-,-\})(W) \subset W_R\otimes V + V\otimes W_R$ where $\phi^\star\{-,-\}^*$ is the pullback by $\phi$  of $\{-,-\}^*$. We also prove in subsection \ref{Kxz sh x}  that the space $\phi(U)=\K\oplus \K \langle x, z \rangle \sh x $ is a coideal for $ \{-,-\}^*$ and therefore $U$ is a coideal for $\phi^\star\{-,-\}^*$. The results of this section will be used to show the main results $(a)$ and $(b)$ announced before.\\\\
In the last section, we prove the main results announced in the previous subsection of the introduction using essentially the results of section \ref{section duality} and section \ref{pullback coihara}. We show (subsection \ref{proof a,b}) that the following statements are equivalent: "$\ls$ is preserved by $\{-,-\}$, "$F$ is a coideal for $\phi^\star\{-,-\}^*$", "$W_R$ is a coideal for $\tilde{\delta}$", "$\D$ is a Lie coalgebra with respect to the formulas defined by Goncharov". This proves $(b)$. To show $(a)$ we prove that $\phi^\star \{-,-\}^*$ induces a Lie cobracket $\delta_{V/F}  : V/F\to V/F\otimes V/F$ and that we have the following diagram of bigraded Lie coalgebra isomorphisms: $ \D \overset{\eta}{\longrightarrow} (W/W_R,\delta) \overset{\bar{i}}{\longrightarrow}(V/F,\delta_{V/F}) \overset{\varepsilon}{\longrightarrow}  (\ls^\vee,\{-,-\}_{\vert \ls}^{\bar{\vee}})$, where $ (\ls^\vee,\{-,-\}_{\vert \ls}^{\bar{\vee}})$ is the bigraded Lie coalgebra dual to $(\ls,\{-,-\})$.\\
In subsection \ref{Subsection Dsh ls D}, we prove results $(c)$ (paragraph \ref{Proof (c)}) and $(d)$ (paragraph \ref{Proof (d)}). To prove $(c)$ we introduce a graded isomorphism $h_m: \K[x_1,\dots,x_m]^\vee \to W_m$ (the subscript $m$ is form the depth $m$ component) and show in proposition \ref{prop hWR}, using corollary \ref{cor ker coeff} of the reminders section, that the image by $h_m$ of the forms orthogonal to $\Dsh$ is $(W_R)_m$. This proves that $W_m/(W_R)_m$ is isomorphic to $\Dsh_m^\vee$ and $(c)$ follows since $W/W_R$ is isomorphic to $\D$. To prove $(d)$ we consider the graded map $h_m: \K[x_1,\dots,x_m]^\vee \to W_m$, the inclusion $i: W_m \to V_m$ and the map $\beta_m:V_m \to \K\langle x,z \rangle_m^\vee$ induced by the pairing $\langle-,-\rangle_\phi$ (the indices $m$ are for depth $m$ components). By theorem \ref{theorem duality}, remark \ref{rmk V/F} and proposition \ref{cor Dsh W} the graded map $\theta=\beta_m\circ i \circ h_m$ induces, with respect to the restriction maps $\K[x_1,\dots,x_m]^\vee \to\Dsh_m^\vee$ and $\K\langle x,z \rangle_m^\vee \to \ls_m^\vee$, an isomorphism $\Dsh_m^\vee\to \ls_m^\vee$. We show using a direct computation that $\theta=f_m^\vee$ and then deduce $(d)$ essentially using an algebraic argument.

\section{Algebra Reminders}\label{Section rappel}
The section contains some algebra reminders. In the first subsection we recall facts on free associative algebras. Most of the material in the first subsection can be found in \cite{Reut}. The second subsection contains basic reminders on bigraded pairings, adjoints of Lie brackets and duality between Lie algebras and Lie coalgebras. In the last subsection, we recall some facts on duality between spaces related by series.
\subsection{Free associative algebras}\label{FA}
We fix a field $\K$ of characteristic zero. Let $A$ be a set and $X$ be the collection of indeterminates $\{x_a\}_{a\in A}$. We denote by $\K\langle X \rangle$ the $\K$-free associative algebra over the indeterminates  $\{x_a\}_{a\in A}$.

\subsubsection{Shuffle product}\label{subsection shuffle}
We recall that for $(p,q)\in \N^2$ a $(p,q)$-shuffle is a permutation $\sigma$ of the set $[1,p+q]$  such that the restrictions $\sigma_{\vert [1,p]}$ and $\sigma_{\vert [p+1,p+q]}$ are increasing functions.  The set of all $(p,q)$-shuffles will be denoted by $S(p,q)$.\\\\
The shuffle product $\sh$ is an associative commutative product on $\K\langle X \rangle$ with unit $1$, given by:
\begin{align}\label{shuffle product}
 x_{a_1}\cdots x_{a_p} \sh x_{a_p+1}\cdots x_{a_p+q}&= \underset{\sigma \in S(p,q)}{\sum}x_{a_{\sigma^{-1}(1)}}\cdots x_{a_{\sigma^{-1}(p+q)}},
\end{align}
for $(a_1,\dots,a_{p+q}) \in A^{p+q} $ and $p,q>0$.
We can also define the shuffle product inductively by the following equation: 
\begin{align}\label{shuffle product12}
w_1x_{a_1} \sh w_2x_{a_2}= (w_1 \sh w_2x_{a_2})x_{a_1} + (x_{a_1}w_1 \sh w_2)x_{a_2},
\end{align}
for $x_{a_1},x_{a_2} \in X$ and $w,w_1,w_2 $ unitary monomials of $\K \langle X\rangle$.\subsubsection{Coproducts}\label{subsection coporducts}
There exists a unique morphism of algebras $\Delta_X : \K\langle X \rangle \to \K\langle X \rangle \otimes \K\langle X \rangle $ for the canonical product of  $\K\langle X \rangle$, called the shuffle coproduct, given by: 
\begin{equation}\label{definition coproduct}
 \Delta_X(x_{a})=1\otimes x_a +x_a \otimes 1, \text{ for } a\in A. 
\end{equation}
One also have the deconcatenation coproduct $\Delta_{\sh}^X: \K\langle X \rangle \to \K\langle X \rangle\otimes \K\langle X \rangle$  which is the linear map defined by: 
\begin{equation}\label{coproduct 2} 
\Delta_{\sh}^X(w)= \sum_{w_1w_2=w} w_1\otimes w_2,  \end{equation}
where $w,w_1$ and $w_2$ unitary monomial. The map $\Delta_{\sh}^X$ is a morphism for the shuffle product.

\subsubsection{Pairings and duality between products and coproducts}\label{pair adj}
The canonical pairing of $\K \langle X \rangle$ is the unique (perfect) pairing $ \langle - , -\rangle_X :   \K \langle X \rangle \otimes \K \langle X \rangle\to \K$ for  which the family of unitary monomials is an orthonormal basis. From this pairing we get a perfect pairing $\langle - , -\rangle_X^{\otimes 2}$ over $\K \langle X \rangle$, given by:
$$\langle w_1 \otimes w_2, w_1'\otimes w_2'\rangle_X^{\otimes 2}=\langle w_1 ,w_1'\rangle_X \langle  w_2,  w_2'\rangle_X,$$
for $w_1,w_2,w_1',w_2' \in \K \langle X \rangle$. The shuffle coproduct $\Delta_X$ is adjoint to the shuffle product and the deconcatenation coproduct $\Delta_{\sh}^X$ is adjoint to the canonical product of $\K \langle X \rangle$:
\begin{equation}\label{coproduct formula}
\langle \Delta_X (w) , w_1\otimes w_2\rangle_X^{\otimes 2}=\langle w, w_1\sh w_2\rangle_X\quad ,\quad \langle \Delta_\sh^X (w) , w_1\otimes w_2\rangle_X^{\otimes 2}=\langle w, w_1w_2\rangle_X, 
\end{equation}
for $w,w_1,w_2 \in \K\langle X\rangle$.
\subsubsection{Derivations}\label{derivations}
Let $d$ be a derivation of the free associative algebra $\K\langle X \rangle$ (for the canonical product). By applying the Leibniz rule we get: 
\begin{equation}\label{derivation rule}
 d(x_{a_1}\cdots x_{a_m})=\underset{i \in [1,m]}{\sum} x_{a_1} \cdots x_{a_{i-1}} d(x_{a_i}) x_{a_{i+1}}\cdots x_{a_m}, 
\end{equation}
for $m \in \N^*$ and $(a_1,\dots,a_{m}) \in A^m $ (if $m=0$, we recover $d(1)=0$). Hence, a derivation of $\K\langle X \rangle$ is entirely determined by its values on the indeterminates $x_a\in X$. Given $\{y_a\}_{a\in A} \subset \K\langle X \rangle$, one checks that the linear mapping $x_{a_1}\cdots x_{a_m}\mapsto \sum_{i \in [1,m]} x_{a_1} \cdots x_{a_{i-1}} y_{a_i}x_{a_{i+1}}\cdots x_{a_m}$ satisfies the Leibniz rule. It follows that:
\begin{proposition}\label{unique derivation}
Let $\{y_a\}_{a\in A}$ be a subset of $\K \langle X \rangle$. There exists a unique derivation $d$ of $\K \langle X \rangle$ sending the indeterminate $x_a$ to $y_a$ for $a\in A$.
\end{proposition}

\subsection{Pairings, adjoints, duality between Lie algebras and Lie coalgebras}\label{duality}
Throughout this subsection $A$ and $B$ are two birgraded $\K$-vector spaces:  
$$ A=\oplus_{m,n\geq 0} A_{m,n} \quad \text{and} \quad B=\oplus_{m,n\geq 0} B_{m,n}.$$
The spaces $A_{m,n}$ and $B_{m,n}$ are called homogenous elements  of bidegree $(m,n)$ of $A$ and $B$ respectively. A pairing $(-,-)_{A,B}: A\otimes B \to \K$ is a bigraded pairing if $(A_{m,n},B_{k,l})=0$ for $(m,n)\neq (k,l)$. Given a bigraded paring $ (-,-)_{A,B}: A\otimes B \to \K$ one can construct a bigraded pairing $ (-,-)_{A,B}^{\otimes 2}: (A^{\otimes2})\otimes( B^{\otimes2}) \to \K$ given by: 
 \begin{equation}\label{paring tens2}
(a\otimes a ',b\otimes b')_{A,B}^{\otimes 2}=(a,b)_{A,B}(a',b')_{A,B},\end{equation}
for $a,a'\in A$ and $b,b' \in B$. We assume that $A_{m,n}$ and $B_{m,n}$ (for $m,n\geq 0$) are finite dimensional, that we have a perfect pairing $(-,-)_{A,B}$ and a bigraded Lie bracket $g: B\otimes B \to B$.\\\\
For a given bigraded vector space $C$ we denote by $C^\vee$  its bigraded dual. Note that if $f$ is a bigraded linear map then the adjoint of $f^*$ with respect to perfect bigraded pairings is also a bigraded map.\\\\
Let $\varphi_{B,2}$ be the canonical isomorphism of bigraded spaces $(B\otimes B)^\vee \to B^\vee \otimes B^\vee$. The space $B^\vee$ equipped with $g^{\overline{\vee}}=\varphi_{B,2} \circ g^\vee$ ($g^\vee$ is the bigraded dual map of $g$) is a bigraded Lie coalgebra and the pair $(B,g^{\overline{\vee}})$ is called the bigraded dual of $(B,g)$. The following proposition is known to be true in the finite dimensional case and the proofs can be adapted to the bigraded case. 
\begin{proposition}\label{dual adjoint of lie bracket}
Let  $g^*:A\to A\otimes A$ be the adjoint of $g$ with respect to the pairings $(-,-)_{A,B}$ and $(-,-)_{A,B}^{\otimes2}$.
\baselineskip 18pt
\begin{itemize}
\item[1)]  The map $g^*$ is a bigraded Lie cobracket.
\item[2)] We have an isomorphism of bigraded Lie coalgebras (i.e. respecting the bigrading) $(A,g^*)\simeq (B^\vee, g^{\overline{\vee}})$ given by $a_{m,n}\mapsto (b_{m,n}\mapsto (a_{m,n},b_{m,n})_{A,B})$, for $a_{m,n} \in A_{m,n}, b_{m,n} \in B_{m,n}$ and $m,n\geq 0$.
\end{itemize}
\end{proposition}

\begin{proposition}\label{Lie pairing}
 Let $C$ be a bigraded vector subspace of $B$.
\baselineskip 18pt
\begin{itemize} \item[1)]
The subspace $C$ is preserved by $g$ if and only if the space $C^\perp\subset A$ orthogonal to $C$ with respect to $(-,-)_{A,B}$ is a coideal for the adjoint $g^*:A\to A\otimes A$ of $g$ with respect to the pairings $(-,-)_{A,B}$ and $ (-,-)_{A,B}^{\otimes 2}$ . By a coideal for $g^*$ we mean $g^* (C^\perp)\subset C^\perp \otimes A+ A\otimes C^\perp$.
\item[2)] Assume that one of the equivalent conditions in $(1)$ is satisfied. The pair $(A/C^\perp,\bar{g}^*)$, where $\bar{g}^*$ is induced by $g^*$, is a bigraded Lie coalgebra and the linear map $(A/C^\perp,\bar{g}^*) \to (C^\vee, g_{\vert C}^{\overline{\vee}}), \bar{a} \to (a,-)_{A,B}$, where $a$ is any lift of $\bar{a}$ to $a$ is an isomorphism of bigraded Lie coalgebras.
\end{itemize}
\end{proposition}
\begin{proof}
Point $(1)$ follows from a general algebra equivalence adapted to our case: $g(C\otimes C)\subset C$ if and only if $g^*(C^\perp) \subset (C\otimes C)^\perp=C^\perp \otimes A+A\otimes C^\perp$. To show $(2)$, we check that $(-,-)_1: A/C^\perp\otimes C \to \K, \bar{a}\otimes c \mapsto (a,c)_{A,B}$, where $a \in A$ is any lift of $\bar{a}$, is a perfect pairing, and that $\bar{g}^*$ is the adjoint of $g_{\vert C}$ with respect to that pairing. We then conclude by applying $(2)$ of proposition \ref{dual adjoint of lie bracket}.
\end{proof}
\subsection{Series and duality} 
Let $A:=\oplus_{m \geq 0} A_m$ and $B:= \oplus_{m \geq 0} B_m$ be graded $\K$-vector spaces with finite dimensional homogeneous elements. Set $A\HT B=\prod_{m\geq 0} (A\otimes B)_m$, where the index $m$ is for degree $m$ component. For $\varphi \in B^\vee$ seen as a linear form over $B$, and $S=\sum_{m\geq 0} S_m \in A\HT B$, with $S_m \in (A\otimes B)_m$, the element $(\mathrm{id}_A\otimes \varphi)(S_m)$ of $A\otimes \K$ is nonzero for only a finite number of $m$. Hence, we can define an element $(\mathrm{id}_A\otimes \varphi )(S)\in A$ by $\sum_{m\geq 0} (\mathrm{id}_A\otimes \varphi)(S_m)$. Now consider the isomorphism $\gamma: A\otimes \K \to A,a\otimes 1 \to  a$. We define for $S\in A\HT B$ a linear map and a space: 
$$\begin{array}{cccc} L_s:& B^\vee &\to & A\\
&\varphi & \mapsto & \gamma((\mathrm{id}_A\otimes \varphi)(S))
\end{array} \qquad \text{and} \qquad A(S)=\mathrm{Im}(L_S) \subset A.$$
If $S\in D(A\HT B)= \prod_{m\geq 0} A_m\otimes B_m \subset A\HT B$ then $L_S$ and $A(S)$ are graded
\begin{remark}
When $B$ is a polynomial ring $A\HT B$ is a space of series with values in $A$. Given $S=\sum_i a_i \otimes P_i \in A\HT B $ with $\{P_i\}_i$ a free family of polynomials the space $A(S)$ is the space generated by the coefficients $\{a_i\}_i$ of $S$. In particular, the space generated by the coefficients of a series is independent of the choice of the free family $\{P_i\}_i$ used to write $S$.
\end{remark}
For $f_A$ and $f_B$ graded maps with sources $A$ and $B$ respectively, denote by $f_A\HT f_B$ the graded morphism of source $A \HT B$ mapping $\sum_{m\geq 0} c_m$, with $c_m\in (A\otimes B)_m$, to $\sum_{m\geq 0}(f_A\otimes f_B)(c_m)$. 
\begin{proposition}
Take $T:B\to B$ a graded linear map and $S\in D(A\HT B)$. The map and the space $L_{(\mathrm{id}_A\HT T)(S)}$ and $A((\mathrm{id}_A\HT T)(S))$ are graded and:
$$L_{(\mathrm{id}_A\HT T)(S)}=L_S\circ T^\vee\quad \text{and} \quad A((\mathrm{id}_A\HT T)(S))=L_S(\mathrm{Ker}(T)'),$$ with $T^\vee: B^\vee\to B^\vee$ the graded dual of $T$ and $\mathrm{Ker}(T)'$ the vector space of $\varphi \in B^\vee$ null over $\mathrm{Ker}(T)$.
\end{proposition}
\begin{proof}
We leave the proof to the reader.
\end{proof}
\begin{corollary}\label{cor ker coeff}
Let $T_1,\dots, T_k$ be graded endomorphisms of $B$ and $S$ an element of $D(A\HT B)$. The graded spaces $L_S((\cap_{i=1}^k \mathrm{Ker}(T_i))')$ and $\sum_{i=1}^k A(\mathrm{id}_A\HT T_i)(S)$ are equal.
\end{corollary}
\section{The linearized double shuffle Lie algebra $\ls$ and the double shuffle space}\label{section ls}
Given a field $\K$ of characteristic zero, we recall in subsection \ref{ssection ls} the definition of the linearized double shuffle Lie algebra $\ls$ (in \cite{Brown} $\ls$ is defined over $\Q$): we introduce the bigraded vector space $\ls\subset \K \langle x,z \rangle$ and then its Lie bracket (Ihara's bracket $\{-,-\}$) originally defined over $\K \langle x,z \rangle$. In subsection \ref{def Dshm} we recall the definition of the double shuffle subspace $\Dsh_m={\oplus}_{k> m \geq 2}\mathrm{Dsh}_{m}({{k}-m})$ of the polynomial algebra $\K[x_1,\dots, x_m]$ introduced in \cite{IKZ} (for $\K=\Q$) for $m\geq2$.

\subsection{The linearized double shuffle Lie algebra $\ls$}\label{ssection ls}
In order to define $\ls$ we use two algebras: the free associative algebra $\K \langle x,z \rangle$ on two  indeterminates $x,z$ and the free associative algebra $\K\langle Y\rangle$ on the set of indetreminates $Y=\{y_i \vert i\in \N^*\}$.\\\\
We define the weight of a monomial $w\in \K \langle x,z \rangle$ as the total degree of $w$ and the depth of $w$ as the degree of $w$ with respect to the variable $z$. The algebra $\K \langle x,z \rangle$ is a bigraded algebra with respect to depth and weight. Similarly, we define the depth of  $w_Y:=y_{n_1}\cdots y_{n_m} \in \K \langle Y\rangle$ as the total degree of $w_Y$ and the weight of $w_Y$ as the number $n_1+\cdots+n_m$.\\\\
The shuffle coproducts of $\K \langle x,z \rangle$ and $ \K \langle Y \rangle$ as defined in (\ref{definition coproduct}) section \ref{Section rappel} will be denoted by $\Delta$ and $\Delta_Y$, respectively. The coproducts $\Delta$ and $\Delta_Y$ respect the bigradings defined above.\\\\
Let  $\pi : \K \langle x,z \rangle \to \K \langle Y \rangle$ be the projection of bigraded spaces, given by: 
\begin{equation}\label{projection pi}
\K\langle x,z \rangle x \mapsto 0, \quad 1\mapsto 1 \quad  \text{and} \quad x^{n_1-1}zx^{n_2-1}z\cdots x^{n_m-1}z\mapsto y_{n_1}y_{n_2}\cdots y_{n_m},\end{equation}
for $m\geq 1$ and $n_1,\dots,n_m\geq 1$.\begin{definition}[\cite{Brown}]\label{definition ls}
The linearized double shuffle space $\ls$  is the bigraded vector space (for  depth and weight) of elements $\psi \in \K \langle x,z \rangle$ such that: 
$$ \Delta(\psi)=1\otimes \psi +\psi \otimes 1\quad, \quad \Delta_Y(\pi(\psi))=1\otimes \pi(\psi)+\pi(\psi) \otimes 1, $$
and the only components of $\psi$ of depth $m\leq 1$ are of odd weight and have depth $1$. 
\end{definition}
For ${k} \geq m\geq 1$, we denote by $\ls_m^{k}$ the depth $m$ and weight $k$ component of $\ls$. We have: $$\ls= \oplus_{k\geq m \geq 1} \ls_m^{k}.$$
Let $[-,-]$ be the canonical Lie bracket of the free associative $\K  \langle x,z \rangle$, i.e $ [w,w']=ww'-w'w$, for $w,w' \in \K \langle x,z \rangle$. Given $w\in \K \langle x,z \rangle$, we denote by $d_w$ the unique derivation of $\K \langle x,z \rangle$ defined by $d_w(x)=0$ and $ d_w(z)=[z,w]$. The existence and uniqueness of $d_w$ follows from proposition \ref{unique derivation} of section \ref{Section rappel}. 
\begin{definition}\label{Ihara bracket}
The Ihara bracket is the bigraded Lie bracket $\{-,-\}: \K \langle x,z \rangle^{\otimes2}\to \K \langle x,z \rangle$ given by: 
$$ \{w,w'\}= d_w(w')-d_{w'}(w)+[w,w'].$$
for $w,w' \in  \K \langle x,z \rangle$.
\end{definition}
It follows form theorem 3.4.3 of \cite{Schnepsari} and its proof that:
\begin{theorem}[\cite{Brown},\cite{Schnepsari}]\label{Ihara bracket old proofs}
The linearized double shuffle space $\ls$ is preserved by the Ihara bracket.
\end{theorem}

\subsection{The double shuffle space $\Dsh_m$}\label{def Dshm}
For $m \geq 2$ and $\sigma \in \mathfrak{S}_m$ we denote by $S_m$ and $T_\sigma$ the automorphisms of the polynomial ring $\K[x_1,\dots,x_m]$ given by: $$ S_m(x_i)= x_i+\cdots  +x_m, \quad T_\sigma(x_i)=x_{\sigma^{-1}(i)}, \quad \text{for $i\in[1,m]$}$$
 and we define for $l\in[1,m-1]$ the linear endomorphisms $T_{m,*}^{(l)}$ and $T_{m,\sh}^{(l)}$ of $\K[x_1,\cdots,x_m]$ by the following:
$$T_{m,*}^{(l)}= \underset{\sigma \in S(l,m-l)}{\sum}{T_\sigma},\quad T_{m,\sh}^{(l)}=T_{m,*}^{(l)} \circ S_m,$$
where $S(p,q)$ denotes the set of $(p,q)$-shuffles, as in the previous sections.
\begin{definition}[\cite{IKZ}]
\baselineskip 18pt
\begin{itemize}
\baselineskip 18pt
\item[1)] The double shuffle subspace $\Dsh_m$ of $\K[x_1,\dots,x_m]$ (for $m\geq 2$) is the intersection of the kernels of the endomorphisms $T_{m,*}^{(l)}$ and $T_{m,\sh}^{(l)}$ for $l\in[1,m-1]$. 
\item[2)] For $m\geq 2$ and $d\geq1$, the double shuffle space $\Dsh_m(d)$ is the vector subspace of polynomials of total degree $d$ lying in $\Dsh_m$.
\end{itemize}
\end{definition}
The maps $T_{m,*}^{(l)}$ and $T_{m,\sh}^{(l)}$ respect the total degree grading of $\K[x_1,\dots, x_m]$. Hence, the total degree induces a grading of $\Dsh_m$. Here, we endow $\Dsh_m$ with a weight grading corresponding to a shift by $+m$ of the total degree grading:
\begin{definition}\label{Dsh grading} 
The weight grading of $\Dsh_m$ is given by:
$$\Dsh_m= {\oplus}_{k\geq m} \Dsh_m^k ,$$ 
where the weight $k$ part $\Dsh_m^k$ of $\Dsh_m$ is the double shuffle space $\Dsh_m(k-m)$. 
\end{definition}

\section{The dihedral Lie coalgebra $\D$}\label{diheral coalgebra}
Given a field $\K$ of characteristic zero, we introduce a free bigraded vector space $W$ and  generating series similar to those in \cite{Gon}. The series are used to define a bigraded subspace $W_R \subset W$. It is shown in proposition \ref{definition dihedral} that we have a natural isomorphism of bigraded spaces $\eta: \D \to W/W_R$. A map $\tilde{\delta}: W \to W\otimes W$ inducing a bigraded Lie cobracket $\delta$ on $W/W_R $ for which $\eta:\D \to W/W_R$ is an isomorphism of bigraded Lie coalgebra (theorem \ref{Lie structure dihedral}), is defined. The construction of $W/W_R$ and $\delta$ is a transcription of the work of \cite{Gon}. We end the section by introducing extensions $V$ and $F$ of $W$ and $W_R$ by a bigraded free vector space $U$. The isomorphic spaces $V/F$ and $W/W_R$ will be used in the next sections.\\\\
For $k\geq m \geq 1$, let $W_{m,k}$ be the free $\K$-vector space with basis the elements $I(n_1,\dots, n_m)$ where $n_i\in \N^*$ and $n_1+\cdots +n_m=k$. We say that $I(n_1,\dots,n_m)$ is of depth $m$ and weight $k$. Let $W$ be the depth-weight bigraded space:
$$ W=\oplus_{k\geq m \geq 1} W_{m,k},$$
with  depth $m$ and weight $k$ component $W_{m,k}$.\\\\
For $m\geq 1$, we define the multivariable series with values in $W$: 
$$ \{ t_1:\cdots : t_m :t_{m+1}\}= \underset{n_1,\dots,n_m\geq 1}{\sum}I(n_1,\dots,n_m) (t_1-t_{m+1})^{n_1-1} \cdots (t_m-t_{m+1})^{n_m-1}$$
$$\text{and} \quad \{t_1,\dots, t_m\}=\{t_1:t_1+t_2: t_1+t_2+t_3 : \cdots :t_1+t_2+\dots+t_m: 0\}.$$
These series are similar to those used by Goncharov (see $(54)$ and $(56)$ in \cite{Gon} for $G=\{e\})$.\\\\
Let $\Wst$ be the subspace of $W$ generated by the coefficients of the series: 
\begin{equation*}
\underset{\sigma \in S(p,q)}{\sum} \{ t_{\sigma^{-1}(1)} : \cdots : t_{\sigma^{-1}(p+q)}: t_{p+q+1}\},\end{equation*}
for all $(p,q) \in (\N^*)^2$, and $\Wsh$ the subspace of $W$ generated by the coefficients of the series: 
\begin{equation*}
\underset{\sigma \in S(p,q)}{\sum} \{ t_{\sigma^{-1}(1)} , \dots , t_{\sigma^{-1}(p+q)}\},
\end{equation*}
for $(p,q) \in (\N^*)^2$.\\\\
We denote by $W_{1,even}$ the subspace of $W$ generated by the elements $I(2n)$ for $n\geq 1$. The spaces $\Wsh,\Wst,$ and $W_{1,even}$ are bigraded subspaces of $W$.

\begin{proposition}\label{definition dihedral}
Set $W_R=\Wst+\Wsh+W_{1,even}$ where $\Wst, \Wsh$ and $W_{1,even}$ are as in the previous paragraph. We have a natural isomorphism of bigraded vector spaces $\eta :  \D  \to W/W_R$. The isomorphism $\eta$ is constructed in the proof below (see (\ref{eta})). 
\end{proposition}
\begin{proof}
In \cite{Gon}, Goncharov introduces a space  $\DGH$ (see page 430 and the space $\DG$ section 4) by generators and relations. For $G=\{e\}$, the generators are the symbols $I_{n_1,\dots,n_m}(e,\dots,e)$ for $m\geq 1$ and $(n_1,\dots,n_m) \in (\N^*)^m$ and the relations can be reduced to: (i) the double shuffle relations and (ii) the distribution relations for $m\geq 1$ and $l=-1$ (following the notations of Goncharov), since the homogeneity relations are empty. 
 It follows from theorem 4.1 section 4.2 of \cite{Gon} that the double shuffle relations imply the inversion relations (a part of the diherdal symmetry relations, see formula (66) section 4.2  \cite{Gon}) for $m\geq 2$, and these relations are exactly the distribution relations for $m\geq 2$ and $l=-1$. Hence, the only relations in $\DGHE$ are the double shuffle relations and the distribution relation for $m=1$ and $l=-1$, and these relations can be readily identified to $W_R$ via the correspondence $I_{n_1,\dots,n_m}(e,\dots,e) \mapsto -I(n_1,\dots,n_m)$. This shows that we have an isomorphism: 
\begin{align}\label{eta}
\begin{split}
\eta: & \widehat{\mathscr{D}}_{\bullet \bullet}(\{e\})=\D \quad \to  \quad \: W/W_R \\
& I_{n_1,\dots,n_m}(e,\dots,e) \to -I(n_1,\dots,n_m).
\end{split}
\end{align}
Since $ I_{n_1,\dots,n_m}(e,\dots,e)$ and $I(n_1,\dots,n_m)$ have the same depth and same weight, $\eta$ is bigraded.
\end{proof}
It follows from theorem 4.1 of \cite{Gon} and the fact that  $W_{1,even}$ is a subset of $W_R$ that: 
\begin{proposition} \label{cyclic rmk}
For $m\geq 1$, the cyclic symmetry relation: 
$$ \{t_1:\cdots:t_{m+1} \} = \{t_{m+1}:t_1\cdots:t_m\}, $$ 
is satisfied in $W/W_R$.
\end{proposition}
For $A$ a vector space $m\geq 1$ and $f$ an element of $A[[t_1,\dots,t_{m+1}]]$ (the space of series in $t_1,\dots, t_{m+1}$ with values in $A$), we define $\text{Cycle}_{m+1}(f)\in A[[t_1,\dots,t_{m+1}]]$ by:
 $$\text{Cycle}_{m+1}(f(t_1,\dots,t_{m+1}))=\overset{m}{\underset{j=0}{\sum}} f(t_{1+j},\cdots,t_{m+1+j})$$ where the indices of the variables $t_{1+k},\cdots,t_{m+1+k}$ are taken modulo $m+1$.\\\\
Let $\tilde{\delta}:W \to W\otimes W$ be the linear map given by:
\begin{equation}\label{Lie cobracket1}
\tilde{\delta}(\{t_1:\cdots:t_{m+1}\})=\sum_{k=2}^{m} \text{Cycle}_{m+1} (\{t_1 :\cdots: t_{k-1} : t_{m+1}\}\wedge \{t_k:\cdots: t_{m+1}\}),
\end{equation} 
for $m\geq 1$ (where $a\wedge b= a\otimes b- b \otimes a$).\\\\
To define the Lie cobracket of $\D$, Goncharov introduces a formula (formula $(80)$, \cite{Gon}) and then shows ($(b)$ of theorem 4.3 of \cite{Gon}) that his formula defines a Lie cobracket over $\D$. The formula defining $\tilde{\delta}$ here, is a transcription of $(80)$ of \cite{Gon} via the correspondence $I_{n_1,\dots,n_m}(e,\dots,e) \mapsto -I(n_1,\dots,n_m)$ inducing the isomorphism $\eta: \D\to W/W_R$ of proposition \ref{definition dihedral}. We deduce from this discussion that:
\begin{remark}\label{rmk DW}
The space $W_R$ is a coideal for $\tilde{\delta}$ if and only if formula $(80)$ of \cite{Gon} defines a Lie cobracket over $\D$ $((b)$ of theorem 4.3 of \cite{Gon}$)$
\end{remark}
and from theorem 4.3 of \cite{Gon} that:
\begin{theorem}\label{Lie structure dihedral}
\baselineskip 18pt
\begin{itemize}
\item[1)] The map $\tilde{\delta}$ induces a map $\delta:  W/W_R \to W/W_R \otimes W/W_R $, providing a bigraded Lie coalgebra structure on $W/W_R$. 
\item[2)] The isomorphism of bigraded spaces $\eta : \D \to W/W_R$ (see proposition \ref{definition dihedral} and formula (\ref{eta}) of its proof)  is an isomorphism of Lie coalgebra with $W/W_R$ equipped with $\delta$ and $\D$ equipped with its Lie cobracket defined in \cite{Gon}.
\end{itemize}
\end{theorem}

For $k\geq m \geq 0$ and $(k,m)\geq 0$, we denote by $U_{m,k}$ the free $\K$ vector space with basis the elements $I'(n_0,n_1,\dots, n_m)$ with $n_i\geq1$ and $n_0+\cdots+n_m=k$, and denote by $U_{0,0}$ the line spanned by $I'(0)$. We say that $I'(n_0,\dots,n_m)$ is of depth $m$ and weight $n_0+\cdots+n_m$ and we denote by $U$ be the depth-weight bigraded space: $$U=\oplus_{k\geq m \geq 0} U_{m,k},$$
with  depth $m$ and weight $k$ component $U_{m,k}$. The extensions $V$ and $F$ ($F\subset V$) of $W$ and $W_R$ are the bigraded spaces for depth and weight:
 $$V=W\oplus U=\oplus_{k\geq m \geq 0} V_{m,k} \quad \text{and} \quad F=W_R\oplus U, $$
where $V_{m,k}=W_{m,k}\oplus U_{m,k}$, with the convention $W_{0,0}=0$. 

\begin{remark}\label{rmk V/F} The natural inclusion $i:W\to V$ induces an isomorphism of bigraded spaces $\bar{i} : W/W_R\to V/F$ and $\bar{i}\circ \eta  : \D \to V/F$ is an isomorphism of bigraded spaces.
\end{remark}

\section{Series with values in the free associative algebra on two indeterminates}\label{Section series}
We fix a field $\K$ of characteristic zero. We define in formula (\ref{Definition Qseries}), for $n\geq 1$, a generating series $Q_n$, in $n$ commuting variables, for the unitary monomials of $\K\langle x,z \rangle$ of depth $n$. We prove (proposition \ref{shuffle Qseries}) a formula describing the shuffle product of two series $Q_n$, $Q_m$ as a sum of similar series. This formula will be used in section \ref{section duality} and is a formal analogue of formula (15) of \cite{Gon1} for generating series of iterated integrals in the case where the iterated integrals represent MZVs.\\\\ 
For $t$ a formal variable we set $\frac{1}{1-xt}:= \underset{n\geq 0}{\sum} x^nt^n$. For $n\geq 0$ we define the series $Q_n(t_{1}, \dots,t_{n})$ in the commuting variables $t_1,\dots,t_n$ with values in $\K\langle x,z\rangle$, by:
\begin{equation}\label{Definition Qseries} Q_0:=1 \quad \text{and}\quad  Q_n(t_{1}, \dots,t_{n}):=\frac{1}{1-xu_n}z \frac{1}{1-xu_{n-1}}z \cdots \frac{1}{1-x u_1}z, \text{ for $n\geq 1$}
\end{equation}
with $u_k=t_{1}+\cdots+t_{k}$ for $k\in[1,n]$. We denote by $\sh$ and $\Delta_\sh$ the shuffle product and the deconcatenation coproduct of $\K\langle x,z\rangle$ respectively.

\begin{lemma}\label{Q series lemma}
Let $t$ and $t'$ be two commuting variables and $Q$ a series with values in $\K\langle x,z\rangle$.
$$\frac{1}{1-xt} \sh (\frac{1}{1-xt'}z Q)= \frac{1}{1-x(t+t')}z(\frac{1}{1-xt} \sh Q)$$

\end{lemma}
\begin{proof} 
One can check that for $a,b,c$ unitary monomials of $\K\langle x,z\rangle$, we have: 
$$a\sh bzc= \underset{a_1a_2=a}{\sum} (a_1\sh b)z(a_2\sh c), \qquad(A)$$
with $a_1$ and $a_1$ unitary monomials. Note that $\Delta_{\sh}(a)= {\sum}_{a_1a_2=a}a_1 \otimes a_2$, with $a_1$ and $a_1$ unitary monomials.
A direct computation shows that $\Delta_{\sh}(\frac{1}{1-xt})=\frac{1}{1-xt}\otimes \frac{1}{1-xt}$. Therefore, we find by applying $(A)$ that: 
$$ \frac{1}{1-xt} \sh (\frac{1}{1-xt'}z Q)= (\frac{1}{1-xt} \sh \frac{1}{1-xt'})z(\frac{1}{1-xt} \sh Q).$$
To prove the lemma one checks that $\frac{1}{1-xt} \sh \frac{1}{1-xt'}=\frac{1}{1-x(t+t')}$.
\end{proof}


\begin{proposition}\label{proposition Qt}
For $n\geq 1$:
$$Q_n(t_{a_1},\dots t_{a_n})=(\frac{1}{1-xt_{a_1}}\sh Q_{n-1}(t_{a_2},\dots t_{a_n})) z.$$
\end{proposition}
\begin{proof}
The proposition is clear for $n=1$. We recall that, for $n>1$, $Q_{n-1}(t_{a_2},\dots t_{a_n})=\frac{1}{1-xv_{n-1}}z \cdots \frac{1}{1-xv_{1}}z$, where $v_k=t_{a_2}+\cdots+t_{a_{k+1}}$, for $k\in[1,n-1]$. Hence, by applying lemma \ref{Q series lemma} iteratively, we find for $n>1$ that:  
\begin{equation}\label{xt sh Q}
\frac{1}{1-xt_{a_1}}\sh Q_{n-1}(t_{a_2},\dots t_{a_n})=\frac{1}{1-xu_{n}}z \cdots \frac{1}{1-xu_{2}}z \frac{1}{1-xu_{1}},
\end{equation}
where $u_1=t_{a_1}$ and $u_k=t_{a_1}+v_{k-1}=t_{a_1}+ \cdots +t_{a_k}$, for $k\geq 2$. The proposition follows.
\end{proof}

\begin{proposition}\label{shuffle Qseries}
For $(p,q)\in (\N^{*})^2$, we have: 
$$ Q_p(t_1,\dots,t_p) \sh Q_q(t_{p+1},\dots,t_{p+q})=\underset{\sigma \in S(p,q)}{\sum} Q_{p+q}(t_{\sigma^{-1}(1)},\dots, t_{\sigma^{-1}(p+q)}).$$
\end{proposition}
\begin{proof}
We show the formula by induction on $p+q$. By proposition \ref{proposition Qt}, we have: $Q_p(t_1,\dots,t_p)=(\frac{1}{1-xt_1} \sh Q_{p-1}(t_2,\dots t_p))z$ and $Q_q(t_{p+1},\dots,t_{p+q})=(\frac{1}{1-xt_1} \sh Q_{q-1}(t_{p+2},\dots t_{p+q}))z$. Hence, by applying formula (\ref{shuffle product12}) of section \ref{Section rappel}, we get:
$$ Q_p(t_1,\dots,t_p) \sh Q_q(t_{p+1},\dots,t_{p+q})= A z+B z ,$$
where
$$A= \frac{1}{1-xt_1}\sh Q_{p-1}(t_2,\dots,t_p) \sh Q_q(t_{p+1},\dots,t_{p+q})$$
and $$ B=\frac{1}{1-xt_{p+1}}\sh Q_{p}(t_1,\dots,t_p) \sh Q_{q-1}(t_{p+2},\dots,t_{p+q}). $$
Now by applying the formula of this proposition for $(p-1,q)$ (or using that $Q_{p-1}=1$ if $p=1$), then using proposition \ref{proposition Qt} (or the convention $S(p-1,q)=\{ \mathrm{id}\}$ if $p=1$) we get:
\begin{align*}
Az&=\sum_{\sigma \in S(p-1,q)} (\frac{1}{1-xt_1}\sh Q_{p+q-1}(t_{\sigma^{-1}(1)+1},\dots,t_{\sigma^{-1}(p+q-1)+1}))z\\
&= \sum_{\sigma \in S(p-1,q)}  Q_{p+q}(t_1, t_{\sigma^{-1}(1)+1},\dots,t_{\sigma^{-1}(p+q-1)+1})=\sum_{\sigma \in S^1(p,q)}  Q_{p+q}(t_{\sigma^{-1}(1)},\dots, t_{\sigma^{-1}(p+q)})
\end{align*}
where $S^1(p,q)\subset S(p,q)$ is the set of shuffles fixing $1$. Similarly, we get:
\begin{align*}
Bz= \sum_{\sigma \in  S^{p+1}(p,q)}  Q_{p+q}(t_{\sigma^{-1}(1)},\dots, t_{\sigma^{-1}(p+q)}).
\end{align*}
Where $S^{p+1}(p,q) \subset S(p,q)$ is the set of shuffles mapping $p+1$ to $1$. Since $S(p,q)=S^1(p,q) \sqcup S^{p+1}(p,q)$, the proposition for the couple $(p,q)$ follows from the formulas for $Az$ and $Bz$ obtained above. We have proved the proposition.
\end{proof}
\section{Orthogonality between $F$ and $\ls$  with respect to $\langle -,-\rangle_\phi$}\label{section duality}
We prove theorem \ref{theorem duality} below stating that $F\subset V$ is the space orthogonal to $\ls$ with respect to a perfect pairing $\langle-,-\rangle_{\phi}: V\otimes \K \langle x,z \rangle \to \K$. In subsection \ref{orthogonal ls}, we give a decomposition (corollary \ref{ls orth}) of the orthogonal complement $\ls^{\perp}$ of $\ls$ with respect to the canonical paring $\langle - , -\rangle$ of $\K \langle x,z \rangle$. In subsection \ref{morphism phi}, we introduce in proposition \ref{morphisme phi} a linear isomorphism of bigraded spaces $\phi : V\to \K \langle x,z \rangle$ and show (corollary \ref{phi F}) that the image of $F$ by $\phi$ is $\ls^{\perp}$. From these results we deduce in subsection \ref{proof orthogonality} that:
\begin{theorem}\label{theorem duality}
\baselineskip 18pt
\begin{itemize}
\item[1)] The space $F$ is the space orthogonal to $\ls$ with respect to the perfect bigraded pairing $\langle-,-\rangle_{\phi}: V\otimes \K \langle x,z \rangle \to \K$ given by: 
$$\langle v,w\rangle_{\phi}=\langle \phi(v),w\rangle$$
for $ v\in V $ and  $w\in \K \langle x,z \rangle$.
\item[2)] Denote by $\ls^\vee$ the bigraded dual of $\ls$. We have an isomorphism of bigraded vector spaces $\varepsilon : V/ F\to \ls^\vee$ given by $\bar{v}\mapsto \langle v,-\rangle_{\phi}=\langle \phi(v),-\rangle$ for $\bar{v}\in V/ F$, where $v$ is any lift of $\bar{v}$ to $V$.
\end{itemize}
\end{theorem}
Note that the theorem implies that the dihedral Lie coalgebra $\D\simeq V/F$ and $\ls^\vee$ are isomorphic bigraded vector spaces.  
\subsection{The orthogonal complement of $\ls$}\label{orthogonal ls} We denote by $\langle - , -\rangle$ and $\langle - , -\rangle_Y$ the canonical perfect pairings (as defined in subsection \ref{pair adj}) of $\K \langle x,z \rangle$ and $ \K \langle Y \rangle$ respectively, and by $\sh$ and $\sh_Y$ the shuffle products (defined in subsection \ref{subsection shuffle}) of $ \K \langle x,z\rangle$ and $\K \langle Y\rangle$ respectively. Let $i:\K \langle Y \rangle\to \K \langle x,z \rangle$ be the linear section to the map $\pi$ introduced in subsection \ref{ssection ls} (equation (\ref{projection pi})) given by $ i(1)=1$ and $ i(y_{n_1}\cdots y_{n_m})=x^{n_1-1}z\cdots x^{n_m-1}z,$ for $k\in \K, m\geq1 \text{ and } (n_1,\dots,n_m) \in (\N^*)^m$, and let $\K \langle x,z \rangle_{{+}}$ and $\K \langle Y \rangle_{{+}}$ be the vector spaces of elements with zero constant term of $\K \langle x,z \rangle$ and $\K \langle Y \rangle$, respectively:
\begin{proposition}\label{primitive and shuffle}
\baselineskip18pt
For $w$ in $\K \langle x,z \rangle$, we have the following equivalences:
\begin{itemize}
\item[1)] $\Delta(w)=w\otimes 1+1\otimes w$ if and only if $w\in (\K \oplus \K \langle x,z\rangle_{{+}}^{\sh2})^{\perp}$,
\item[2)] $\Delta_Y(\pi(w))=\pi(w)\otimes 1+1\otimes \pi(w)$ if and only if $w\in i(\K\oplus \K \langle Y \rangle_{{+}}^{\sh_Y 2} )^\perp$,
\end{itemize}
where $A^{\perp}$ is the orthogonal complement of $A$ with respect to $\langle -,- \rangle$ for $A\subset \K \langle x,z \rangle$ and $B^{\ast 2}= B \ast B$ for $\ast$ a given product.
\end{proposition}
\begin{proof}
The result follows from formula (\ref{coproduct formula}) section \ref{Section rappel} and the fact that $ \langle a , b\rangle_Y=\langle i(a), i(b)\rangle$.
\end{proof}

\begin{corollary}\label{ls orth}
We can decompose the orthogonal complement $\ls^\perp$ of $\ls$ with respect to $\langle-,-\rangle$ into:
$$ \ls^\perp=\K + \K \cdot x +K_{1,even}+ \K \langle x,z\rangle_{{+}}^{\sh2} +i(\K \langle Y \rangle_{{+}}^{\sh_Y 2}), $$
where $K_{1,even} \subset \K \langle x,z\rangle$ is the vector space of elements of depth $1$ and even weight.
\end{corollary}
\begin{proof}
One uses the previous proposition and that the depth zero component $K_0$ of $\K \langle x,z\rangle$ is $K_0=\K \oplus \K[x]=\K\oplus \K \cdot x \oplus (x\sh \K[x]_+)$, where $\K[x]_+$ is the vector subspace of $\K[x]$ of polynomials with no constant term. 
\end{proof}
\subsection{A map $\phi$ sending $F$ to the orthogonal complement of $\ls$}\label{morphism phi}
We recall that the vector space $V=W\oplus U$ introduced in section \ref{diheral coalgebra} is freely generated by the elements $I(n_1,\dots,n_m)$ (generating $W$) and the elements $I'(n_0,n_1,\dots,n_m)$ (generating $U$).
\begin{proposition}\label{morphisme phi}
The linear map $\phi : V \to \K\langle x,z\rangle$ given by: 
$$ \phi(I(n_1,\dots,n_m))=x^{n_m-1}z\cdots x^{n_1-1}z\quad , \quad \phi(I'(n_0,n_1,\dots,n_m))=x^{n_m-1}z\cdots x^{n_1-1}z x^{n_0-1} \sh x,$$
for $m\geq1, (n_0,\dots,n_m) \in (\N^*)^{m+1}$ and 
$$ \phi(I'(n_0))=x^{n_0-1}\sh x \quad , \quad \phi(I'(0))=1,$$ 
for $n_0 \in N^* $ is an isomorphism of bigraded vector spaces. 
\end{proposition}
\begin{proof}
The family $\{1\} \cup \{x^{n_0}\}_{n_0 \geq 1} \cup \{ x^{n_m-1}z\cdots x^{n_1-1}z , x^{n_m-1}z\cdots x^{n_1-1}z x^{n_0}\}_{m\geq 1, n_i \geq 1} $ is a basis of $\K\langle x,z \rangle$. We deduce from this and the following:
$$ x^{k} \sh x=\frac{x^{k+1}}{k+1}\quad, \quad wzx^{k}\sh x=(k+1)wzx^{k+1}+(w\sh x)zx^k, $$
for $k\geq 0$ and $w\in  \K\langle x,z \rangle$, that the family  $$\{1\} \cup \{x^{n_0-1}\sh x\}_{n_0 \geq 1} \cup \{ x^{n_m-1}z\cdots x^{n_1-1}z , x^{n_m-1}z\cdots x^{n_1-1}z x^{n_0-1}\sh x\}_{m\geq 1, n_i \geq 1} $$ is a basis for $\K\langle x,z \rangle$. The proposition follows. 
\end{proof}

\begin{corollary}\label{xy decomposition}
We can uniquely decompose an element $w\in \K\langle x,z\rangle_{{+}}$ into a sum:
$$w=w_1z+w_2\sh x, $$
with $w_1$ and $w_2$ in $\K\langle x,z\rangle$. 
\end{corollary}
\begin{notation}\label{series ()}
For $m\geq 1$, we denote by $(t_1:\cdots : t_{m+1})$ and $(t_1, \dots , t_m)$ the images by $\phi$ of the series $\{t_1:\cdots : t_{m+1}\}$ and $\{t_1, \dots , t_m\}$ (defined in section \ref{diheral coalgebra}) respectively.
\end{notation}
 \begin{remark}\label{() frac}
For $m\geq 1$: $$(t_1:\cdots : t_{m+1}):=\frac{1}{1-x(t_{m}-t_{m+1})}z\cdots\frac{1}{1-x(t_{1}-t_{m+1})}z.$$
\end{remark}
\begin{proposition}\label{Image Vi}
Let $\phi$ be the isomorphism of proposition \ref{morphisme phi}.
\baselineskip18pt
\begin{itemize}
\item[1)] The image of $\Wst$ with respect to $\phi$ is $i(\K \langle Y\rangle_{{+}}^{\sh_Y2})$.
\item[2)] The image of $\Wsh$ with respect to $\phi$ is $ (\K \langle x,z\rangle z)^{\sh2}$.
\item[3)] The image of $\Wsh+U$ with respect to $\phi$ is $ \K\oplus\K\cdot x\oplus \K \langle x,z\rangle_{{+}}^{\sh2}$.
\end{itemize}
\end{proposition}
\begin{proof}
 We prove (1). By definition of $\Wst$, $\phi(\Wst)$ is generated by the coefficients of the series:
\begin{align*}M_{p,q}&:=\underset{\sigma \in S(p,q)}{\sum} ( t_{\sigma^{-1}(1)} : \cdots : t_{\sigma^{-1}(p+q)} :t_{p+q+1})\\
  & = \underset{\sigma \in S(p,q)}{\sum}\: \: \underset{n_1,\dots,n_m\geq 1}{\sum} i(y_{n_m}\cdots y_{n_{m-1}}\cdots y_{n_1}) (t_{\sigma^{-1}(1)}-t_{m+1})^{n_1} \cdots (t_{\sigma^{-1}(m)}-t_{m+1})^{n_m}\\
&=\underset{\sigma \in S(p,q)}{\sum} \: \: \underset{k_1,\dots,k_m\geq 1}{\sum}  i(y_{k_{\sigma^{-1}(m)}}y_{k_{\sigma^{-1}(m-1)} }\cdots y_{k_{\sigma^{-1}(1)}} )(t_1-t_{m+1})^{k_1} \cdots (t_m-t_{m+1})^{k_m},\\
\end{align*} 
for $(p,q) \in (\N^*)^2$. One checks using formula (\ref{shuffle product}) section \ref{Section rappel} that:
$$ \underset{\sigma \in S(p,q)}{\sum} y_{k_{\sigma^{-1}(m)}}\cdots y_{k_{\sigma^{-1}(m-1)} }=y_{k_{m}}\cdots y_{k_{p+1}}\sh y_{k_{p}}\cdots y_{k_1}. $$
This proves that $\phi(\Wst)=i(\K \langle Y\rangle_{{+}}^{\sh_Y2})$. We proved (1). We now show (2). To compute $\phi(\Wsh)$ we need to describe the coefficients of 
\begin{align*}
N_{p,q}&:=\underset{\sigma \in S(p,q)}{\sum} (t_{\sigma^{-1}(1)} , \dots , t_{\sigma^{-1}(p+q)})\\
&= \underset{\sigma \in S(p,q)}{\sum} Q_{p+q}(t_{\sigma^{-1}(1)},\dots, t_{\sigma^{-1}(p+q)}),
\end{align*}
where $Q_{p+q}$ is the series as in of formula (\ref{Definition Qseries}) of section \ref{Section series}. Hence, by proposition \ref{shuffle Qseries}: 
\begin{align*}
N_{p,q}^{\phi}& =  Q_p(t_1,\dots,t_p) \sh Q_q(t_{p+1},\dots,t_{p+q})\\ \label{Npq}
&= \underset{n_1,\dots,n_m\geq 0}{\sum} (x^{n_p}z\cdots x^{n_1}z \sh x^{n_{p+1}}z\cdots x^{n_{p+q}}z)R^p(n_1,\dots,n_{p+q}),
\end{align*}
where $R^p(n_1,\dots,n_{p+q})=t_1^{n_1}\cdots (t_1+\cdots+t_p)^{n_p}t_{p+1}^{n_{p+1}}\cdots (t_{p+1}+\cdots+t_{p+q})^{n_{p+q}}$. Since the monomials $R^p(n_1,\dots,n_{p+q})$ are linearly independent, we deduce that $\phi(\Wsh)= (\K \langle x,z\rangle z)^{\sh2}$. We still have to prove (3). One checks using the definition of $U$ and $\phi$ that 
\begin{equation}\label{phiU}
\phi(U)= \K \oplus \K\cdot x \oplus (\K \langle x,z\rangle_+ \sh x).
\end{equation}
 On the other hand, a direct computation using corollary \ref{xy decomposition} shows that 
$$\K \langle x,z\rangle_{{+}}^{\sh 2}=\K \langle x,z\rangle_{{+}} \sh x +(\K \langle x,z\rangle z)^{\sh 2}.$$
We deduce (3) from (2) and the last two equations.
\end{proof}

\begin{corollary}\label{phi F}
 The image of $F=W_R\oplus U$ with respect to $\phi$ is $\ls^{\perp}$.
\end{corollary} 
\begin{proof}
We recall that $F=W_R\oplus U= \Wst +\Wsh+W_{1,even}+U$. The space $\phi(W_{1,even}\oplus U_{1,even})$, where $U_{1,even}$ is the space of elements of  depth $1$ and even weight of $U$, is equal to $K_{1,even}$ of proposition \ref{ls orth} for dimensional reasons since $\phi$ is an isomorphism. The corollary follows from this, proposition \ref{ls orth} and the previous proposition.
\end{proof}
\subsection{Proof of the orthogonality between $F$ and $\ls$}\label{proof orthogonality}
We recall that the pairing $ \langle -,-\rangle_\phi: V\otimes \K \langle x,z\rangle\to \K$ of theorem \ref{theorem duality} is given by $\langle v,w\rangle_{\phi}=\langle \phi (v),w\rangle$ for $ v\in V $ and  $w\in \K \langle x,z \rangle$. Theorem \ref{theorem duality} is an immediate consequence of corollary \ref{phi F}, the fact that $\phi$ is a bigraded isomorphism and that $\langle -, -\rangle$ is a perfect bigraded pairing.

\section{The dihedral Lie cobracket and the pullback by $\phi$ of the Ihara cobracket}\label{pullback coihara}
We have seen that $\K \langle x,z \rangle$ is equipped with a perfect pairing $\langle -,- \rangle$. We can also equip $\K \langle x,z \rangle^{\otimes 2}$ with a perfect pairing $\langle -,- \rangle^{\otimes 2}$ naturally obtained out of $\langle -,- \rangle$. In  subsection \ref{Ih Co}, we compute the adjoint $\{-,-\}^*$ of $\{-,-\}$ with respect to these pairings. The cobracket $\{-,-\}^*$ can be identified to the bigraded dual of the Ihara bracket via the isomorphism $\K\langle x,z \rangle\simeq \K\langle x,z \rangle^\vee $ induced by $\langle -,- \rangle$. For that we introduce different operators allowing us to decompose $\{-,-\}^*$ (definition \ref{def di} and proposition \ref{decomp coIhara}). We then give formulas describing these operators (proposition \ref{di*}), allowing us to compute (lemma \ref{lemma dP}) the image by these operators of a generating series $P(v_0,\dots,v_m)$ (defined by (\ref{serie Pv}) ) of the unitary monomials of depth $m$ of $\K \langle x,z \rangle$, for $m\geq 0$. This gives (proposition \ref{coIh P}) an explicit formula for $\{-,-\}^*(P(v_0,\dots,v_m))$. In subsection \ref{sec comp Ih delta}, we deduce (proposition \ref{comp Ih delta}) using the formula for $\{-,-\}^*(P(v_0,\dots,v_m))$ that $(\tilde{\delta} -\phi^\star \{-,-\})(W) \subset W_R\otimes V + V\otimes W_R$, where $\tilde{\delta}:W\to W\otimes W \subset V\otimes V$ is the map introduced in subsection \ref{diheral coalgebra} and $\phi^\star\{-,-\}^*$ is the pullback of $\{-,-\}^*$ by the map $\phi$ of section \ref{section duality}. We also show in subsection \ref{Kxz sh x} (corollary \ref{M*Ihara Kxzx}) that the space $\phi(U)=\K\oplus \K \langle x, z \rangle \sh x $ is a coideal for $ \{-,-\}^*$ and therefore $U$ is a coideal for $\phi^\star\{-,-\}^*$. To do so we first show that $\K \langle x, z \rangle \sh x $ is a coideal for each one of the operators appearing in the decomposition of $\{-,-\}^*$. The results of this section will be used to show in the next section the main results $(a)$ and $(b)$ announced in the introduction of the paper.
\subsection{The Ihara cobracket}\label{Ih Co}
In this subsection, adjoints will be taken with respect to the pairings $\langle -,- \rangle$ and $\langle -,- \rangle^{\otimes 2}$  (obtained out of $\langle -,- \rangle$ by (\ref{paring tens2}), section \ref{Section rappel}). The adjoint of a linear map $f$ is denoted by $f^*$.
\begin{definition}\label{def di}
\baselineskip 18pt
\begin{itemize} 
\item[1)]
For $w=x^{n_m}z\cdots x^{n_1}zx^{n_0}$ and $1 \leq i $, we define four monomials
$w_{i,+}^R,w_{i,+}^L$ and $w_{i,-}^R,w_{i,+}^L$:
\begin{align*}
 w_{i,+}^R=x^{n_{m-i}}z\cdots x^{n_1}zx^{n_0} \quad &, \quad w_{i,+}^L=x^{n_m}z\cdots x^{n_{m-i+1}}z,\\\\
 w_{i,-}^R=zx^{n_{i-1}}z\cdots x^{n_1}zx^{n_0} \quad &, \quad w_{i,-}^L=x^{n_m}z\cdots x^{n_{i+1}}zx^{n_i}, \qquad \text{if $m\geq i$}, \end{align*}
 and $$w_{i,+}^R=w_{i,+}^L=w_{i,-}^R=w_{i,-}^L=0, \quad \text{otherwise.}$$ 

\item[2)]
For $i \geq 1$, we define the linear maps $d_{i,+},d_{i,-}$ from $\K\langle x,z \rangle^{\otimes 2}$ to $ \K\langle x,z \rangle$ by the following: 
$$d_{i,+}(w'\otimes w)=w_{i,+}^Lw' w_{i,+}^R\quad, \quad d_{i,-}(w'\otimes w)=w_{i,-}^L w'w_{i,-}^R $$
for $w,w'$ unitary monomials of $\K\langle x,z \rangle$.
\end{itemize}
\end{definition}
\begin{remark}\label{rmk d}
Consider the linear map $d: \K\langle x,z \rangle^{\otimes 2}\to \K\langle x,z \rangle,w\otimes w'\mapsto d_w(w')$ where $d_w$ is the derivation of $\K\langle x,z \rangle$ given by $d_w(x)=0,d_w(z)=[z,w]$ (as in subsection \ref{ssection ls}). We can show using formula $(\ref{derivation rule})$ section \ref{Section rappel} that:
$$d(w'\otimes w)= \underset{1\leq i \leq m}{\sum} (d_{i,+} -d_{i,-})(w' \otimes w),$$
for $w\otimes w'$ of depth $m\geq 0$.
\end{remark}

Denote by $\tau$ the involution of $\K\langle x,z \rangle\otimes \K\langle x,z \rangle$ sending $w\otimes w'$ to $w'\otimes w$, by $\Delta_\sh$ the deconcatenation coproduct of $\K\langle x,z\rangle$ (see (\ref{coproduct 2}), section \ref{Section rappel}) and by $\mu:\K\langle x,z \rangle^{\otimes2}\to\K\langle x,z \rangle$ the product of the associative algebra $\K\langle x,z \rangle$.

\begin{proposition}\label{decomp coIhara}
For $w\in \K\langle x,z \rangle$ of depth $m\geq 0$, we have:
\begin{align*}
\{-,-\}^*(w)=(\mathrm{id}-\tau) \circ( \Delta_{\sh} + \underset{1\leq i \leq m}{\sum} d_{i,+}^*-d_{i,-}^* ) (w), 
\end{align*}
where $\tau$ and $\Delta_{\sh}$ are as in the previous lemma and $d_{i,+},d_{i,-}$ are the maps of definition \ref{def di}.
\end{proposition}
\begin{proof}
We have seen in remark \ref{rmk d} that, for $w\otimes w'$ of depth $m$: $d(w'\otimes w)= \sum_{1\leq i \leq m} (d_{i,+} -d_{i,-})(w' \otimes w)$. Since the linear maps $d,d_{i,+},d_{i,-}$ (for $i\geq 1$) are depth and weight graded maps, we have: $d^*(w)= \sum_{1\leq i \leq m} (d_{i,+}^*-d_{i,-}^* )(w)$, for $w$ of depth $m$. On the other hand $\mu^*=\Delta_{\sh}$  by formula (\ref{coproduct formula}) of section \ref{Section rappel}). The proposition follows from these equations and the fact that $\{-,-\}=(\mu+d)(\mathrm{id}-\tau)$.
\end{proof}

\begin{proposition}\label{di*}
For $i \geq 1$ and $w$ a unitary monomial of $\K\langle x,z \rangle_{{+}}$, we have:
\[
d_{i,+}^*(w)=(1\otimes w_{i,+}^L)\Delta_{\sh}(w_{i,+}^R) \quad\text{and}\quad
d_{i,-}^*(w)=\Delta_{\sh}^{op}(w_{i,-}^L)(1\otimes w_{i,-}^R),\]
where $d_{i,+}, d_{i,-}, w_{i,+}^L,w_{i,+}^R,w_{i,-}^L$ and $w_{i,-}^R$ are as in definition \ref{def di}.
\end{proposition}
\begin{proof}
We have $d_{i,+}^*(w)=\sum_{(u, v) \in E} u\otimes v$, where $E$ is the set of couples of unitary monomials $(u,v)$ such that $d_{i,+}(u\otimes v)=w$. The last condition is equivalent to the conditions $v_{i,+}^L=w_{i,+}^L$ and $uv_{i,+}^R=w_{i,+}^R$. Therefore,
$$ d_{i,+}^*(w)=\sum_{u'v'=w_{i,+}^R} u'\otimes w_{i,+}^Lv'=(1\otimes w_{i,+}^L)\Delta_{\sh}(w_{i,+}^R).$$
This proves the first equation. One can prove the other equationt similarly. We leave the proof to the reader.
\end{proof}

For $m\geq 0$, we consider the generating series of depth $m$ unitary monomials given by:
 \begin{equation}\label{serie Pv}
P(v_0,\dots,v_m):=\frac{1}{1-xv_{m}}z\cdots\frac{1}{1-xv_{1}}z\frac{1}{1-xv_{0}},
\end{equation}
if $m\geq 1$ and $P(v_0,\dots,v_m)=\frac{1}{1-xv_{0}}$ if $m=0$.
We note that for $k\in[1,m-1]$:
\begin{equation}\label{prod p}
P(0,v_{k+1},\dots v_m) P(v_0,\dots,v_k)=P(v_{k+1},\dots v_m)P(v_0,\dots,v_k,0)=P(v_0,\dots,v_m).
\end{equation}
\begin{lemma}\label{lemma dP}
\begin{itemize}
\item[1)] For $m\geq 0$:
$$ \Delta_\sh(P(v_0,\dots,v_m))= \sum_{i=0}^{m} P(v_i,\dots, v_m) \otimes  P(v_0,\dots,v_i)$$
\item[2)] For $m\geq i \geq 1$:
$$ d_{m-i+1,+}^*(P(v_0,\dots,v_m))= \sum_{k=0}^{i-1} P(v_k,\dots, v_{i-1})\otimes P(v_0,\dots v_k,v_{i},\dots, v_m), $$
\item[3)] For $m\geq i \geq 1$:
$$ d_{i,-}^*(P(v_0,\dots,v_m))=\sum_{k=i}^m P(v_i,\dots,v_k)\otimes P(v_0,\dots, v_{i-1},v_k,\dots,v_m). $$
\end{itemize}
\end{lemma}
\begin{proof}
We first prove $(1)$. One checks using the definition of $\Delta_\sh$ that $\Delta_\sh(azb)=\Delta_\sh(a)(1\otimes zb)+ (az\otimes 1)\Delta_\sh(b)$. By iterating this formula we get:
$$ \Delta_\sh(a_mz\cdots a_1za_0)=\sum_{i\in [0,m]} (a_1 z\cdots a_{i-1}z \otimes 1)     \Delta_\sh(a_i) (1\otimes a_iz \cdots a_1z a_0).$$
We find $(1)$ for $m\geq 1$, by applying the last formula for $a_i=\frac{1}{1-xv_i}$ and using the formula $\Delta(\frac{1}{1-xv_i})=\frac{1}{1-xv_i}\otimes \frac{1}{1-xv_i}$ seen previously and corresponding to the case $m=0$ of $(1)$. We have proved $(1)$. We now prove $(2)$. We have (see definition \ref{def di}): 
$$ P(v_0,\dots,v_m)_{m-i+1,+}^L= P(0,v_{i},\dots, v_m) \quad \text{and} \quad  P(v_0,\dots,v_m)_{m-i+1,+}^R= P(v_0,\dots, v_{i-1}).$$
Therefore by proposition \ref{di*}:
\begin{align*}
d_{m-i+1,+}^*(P(v_0,\dots,v_m))&=(1\otimes P(0,v_{i},\dots, v_m)) \Delta_\sh(P(v_0,\dots, v_{i-1}))\\
&\overset{(*)}{=} \sum_{k=0}^{i-1} P(v_k,\dots, v_{i-1})\otimes P(0,v_{i},\dots, v_m) P(v_0,\dots v_k) \\
&\overset{(**)}{=} \sum_{k=0}^{i-1} P(v_k,\dots, v_{i-1})\otimes P(v_0,\dots v_k,v_{i},\dots, v_m),
\end{align*}
where we apply $(1)$ of this proposition for $(*)$ and equation (\ref{prod p}) in the paragraph before the proposition for $(**)$. Point $(2)$ is proved. We prove $(3)$ similarly using: $$ P(v_0,\dots,v_m)_{i,-}^L= P(v_{i},\dots, v_m) \quad \text{and} \quad  P(v_0,\dots,v_m)_{i,-}^R= P(v_0,\dots, v_{i-1},0),$$  
proposition \ref{di*}, $(1)$ of this proposition and equation (\ref{prod p}).
\end{proof}

\begin{proposition}\label{coIh P}
Take $m\geq 0$ and $i,k\in[0,m]$ and set: 
$$ P_{k,i}(v_0,\dots,v_m)=\begin{cases}
 P(v_k,\dots, v_m) \wedge P(v_0,\dots,v_k) & \text{if $i=0$}\\
 P(v_0,\dots, v_{i-1},v_k,\dots,v_m) \wedge P(v_i,\dots,v_k)  & \text{if $k\geq i \geq 1$}\\
  P(v_k,\dots, v_{i-1})\wedge P(v_0,\dots v_k,v_{i},\dots, v_m) & \text{otherwise}
\end{cases}.$$
We have $\{-,-\}^*(P(v_0,\dots,v_m))= \sum_{i,k\in[0,m]} P_{k,i}(v_0,\dots,v_m)$.
\end{proposition}
\begin{proof}
The proposition is a consequence of the previous lemma and the decomposition of $\{-,-\}^*$ in proposition \ref{decomp coIhara}.
\end{proof}
\begin{corollary}\label{depth 1}
$\{-,-\}^*(P(v_0,v_1))=0$
\end{corollary}
\begin{proof}
This follows from the proposition, since $P_{0,0}(v_0,v_1)=-P_{1,0}(v_0,v_1)=P(v_1)\wedge P(v_0,v_1)$ and  $P_{1,1}(v_0,v_1)=-P_{0,1}(v_0,v_1)=P(v_0)\wedge P(v_0,v_1)$. 
\end{proof}

\subsection{Comparison between the dihedral cobracket and the Ihara cobracket}\label{sec comp Ih delta}

\begin{notation}\label{notation}
 Let $T$ and $T'$ be series with values in a vector space $A$ (or $A^{\otimes 2} $) and $B$ a vector subspace of $A$. We use the notation $T\equiv_B T'$ if $T-T'$ is a series with values in $B$ (or $B\otimes A + A\otimes B$).
\end{notation} 
Let $\phi:V \to \K\langle x,z \rangle $ be the isomorphism of proposition \ref{morphisme phi}. By definition of $U\subset V$, $I'(s)\in U$ (for $s\geq 0$), and by definition $\phi(I'(s))=(s+1)x^s$ (for $s\geq 0$). We recall (see notation \ref{series ()} and remark \ref{() frac}) that:
 $$(t_1:\cdots:t_{m+1})=\phi(\{t_1:\cdots:t_{m+1}\})=\frac{1}{1-x(t_{m}-t_{m+1})}z\cdots\frac{1}{1-x(t_{1}-t_{m+1})}z.$$
\begin{lemma}\label{phi inv P}
 For $m \geq 1$, $$\phi^{-1} (P(0,v_1-v_{m+1},\dots, v_m-v_{m+1}))=\{v_1:\cdots:v_{m+1}\},$$ 
$$\phi^{-1}(P(v_{0},v_1,\dots ,v_m))\equiv_U \{v_1:\cdots :v_m:v_{0} \},$$
 and $\phi^{-1}(P(v_{0}))\equiv_U 0$.
\end{lemma}
\begin{proof}
The first equation follows from the equation preceding the proposition since $P(0,v_1-v_{m+1},\dots, v_m-v_{m+1})=(v_1:\cdots:v_{m+1})$. From the discussion preceding the proposition we deduce that $P(v_0)=\frac{1}{1-xv_0}$ is with values in $\phi(U)$ and hence $\phi^{-1}(P(v_{0}))\equiv_U 0$. We have proved the third equation. Let us prove the second equation. Take $m\geq 1$. We will use the series $Q_m$ introduced in section \ref{Section series}. We have: $ P(v_{0},v_1,\dots ,v_m)=Q_m(v_1,v_2-v_1,\dots,v_{m}-v_{m-1})\frac{1}{1-xv_{0}}$. By applying (\ref{xt sh Q}) section \ref{Section series} for $n={m+1}, t_{a_1}=v_{0},t_{a_2}=v_1-v_{0}$ and $t_{a_k}=v_{k-1}-v_{k-2}$ (for $k\in [3,n]$), we get:
$$P(v_{0},v_1,\dots ,v_m)=Q_m(v_1-v_{0},v_2-v_1,\dots,v_{m}-v_{m-1}) \sh \frac{1}{1-xv_{0}}.$$
Since $\frac{1}{1-xt}=1+x\sh \sum_{k\geq0} \frac{x^k t^{k+1}}{k+1}$ and $Q_m(v_1-v_{0},v_2-v_1,\dots,v_{m}-v_{m-1}) =(v_1:\cdots:v_m: v_{0})$: 
$$P(v_{0},v_1,\dots ,v_m)=(v_1:\cdots:v_m: v_{0})+x\sh \theta,$$
for a given series $\theta$. Finally, $(x\sh \theta) \in \phi(U)=\K\oplus (\K\langle x,z \rangle \sh x) $ (see (\ref{phiU}) subsection \ref{morphism phi} ) and $\phi^{-1}((v_1:\cdots:v_m: v_{0}))=\{v_1:\cdots:v_m: v_{0}\}$ by definition. This proves the second equation. We have proved the proposition.
\end{proof}
We recall that the pullback $f^\star Z:A \to A\otimes A$ of a linear map $Z: B \to B\otimes B$ by an isomorphism of vector spaces $f:A\to B$ is given by:
\begin{equation*}
f^\star Z=(f^{-1} \otimes f^{-1}) \circ Z \circ f,
\end{equation*}
and that $\tilde{\delta}: W\to W\otimes W \subset V\otimes V$, is the map defined in section \ref{diheral coalgebra} by:
\begin{equation*}
\tilde{\delta}(\{t_1:\cdots:t_{m+1}\})= \sum_{k=2}^{m} \overset{m}{\underset{j=0}{\sum}} \gamma_{kj} ,
\end{equation*} 
with $\gamma_{k,j}=\{t_{1+j} :\cdots: t_{k+j-1} : t_{m+j+1}\}\wedge \{t_{k+j}:\cdots: t_{m+j+1}\})$ where the indices of the variables $t_{1+j},\cdots,t_{m+1+j}$ are modulo $m+1$ and are taken in $[1,m+1]$
\begin{proposition}\label{comp Ih delta}
We have $(\tilde{\delta} -\phi^\star \{-,-\})(W) \subset W_R\otimes V + V\otimes W_R$. The same is true if we replace $W_R$ by $F$. 
\end{proposition}
\begin{proof}
Take $w\in W$ of depth $m=1$, $\tilde{\delta}(w)=0$ and $\phi(w)$ is of depth $1$ hence by corollary \ref{depth 1} $\{-,-\}^*(\phi(w))=0$. This proves the proposition for depth $1$ elements of $W$. We now prove the proposition for depth $m>1$ elements. Take $m>1$. We have $P(0,t_1-t_{m+1},\dots,t_m-t_{m+1})= (t_1:\cdots:t_{m+1})=\phi(\{t_1:\cdots:t_{m+1}\})$. Therefore, by proposition \ref{coIh P}: 
$$\phi^\star\{-,-\}(\{t_1,\dots:t_{m+1}\})=\sum_{k,i\in [0,m]} \theta_{k,i}.$$
where $\theta_{k,i}= \phi^{-1}\otimes \phi^{-1}(\{-,-\}^*(P_{k,i}(0,t_1-t_{m+1},\dots,t_m-t_{m+1})))$.
Using lemma \ref{phi inv P}, we get:
$$ \theta_{k,i}=\begin{cases}
 s_{k,0}s_{k,m} \{t_{k+1}:\cdots: t_m:t_k\}\wedge \{t_1,\cdots :t_k:t_{m+1}\} & \text{if $i=0$}\\
s_{k,i}\{t_{1} :\cdots:t_{i-1}:t_k:\cdots:t_{m+1}\} \wedge \{t_{i+1}:\cdots:t_k:t_i\}  & \text{if $k\geq i \geq 1$}\\
  s_{k,i-1}\{t_{k+1}:\cdots: t_{i-1}: t_{k}\}\wedge \{t_1:\cdots: t_k:t_{i}:\cdots : t_m\} & \text{otherwise}
\end{cases},$$
where $s_{a,b}=0$ if $a=b$, $1$ otherwise. One checks that for $k,i\in[0,m]$ such that $s_{k,m}s_{k,i}s_{k,i-1} \neq 0$ (these cases are those where the formula above is not trivial),  $\theta_{k,i}$ is equal to $\gamma_{m+i-k+1,k}$ up to cyclic symmetry (proposition \ref{cyclic rmk}) where the indices of $\gamma_{m+i-k+1,k}$ are modulo $m+1$ and are taken in $[0,m]$. Hence, by proposition \ref{cyclic rmk},  $(\theta_{k,i}- \gamma_{m+i-k+1,k})\in (W_R\otimes V +V\otimes W_R)$ if $s_{k,m}s_{k,i}s_{k,i-1} \neq 0$. Finally, the number of $\theta_{k,i}$ in the formula for $\phi^\star\{-,-\}(\{t_1,\dots:t_{m+1}\})$ above is equal to $m^2$, the number of $(i,k)\in[0,m]^2$ satisfying $s_{k,m}s_{k,i}s_{k,i-1} = 0$ is $2m$ and the number of $\gamma_{k,j}$ in the formula for $\tilde{\delta}$ is equal to $m^2-2m$. This proves the first part of the proposition since the mapping $\Z^2\to \Z^2, (k,i)\mapsto (m+i-k+1,k)$ is a bijection. One can replace $W_R$ with $F$ since $W_R\subset F$. We have proved the proposition.
\end{proof}

\subsection{A coideal for the Ihara cobracket}\label{Kxz sh x}

\begin{proposition}\label{prop wshx}
Take $i\geq 1$ and Let $L$ be one of the operators $\Delta_\sh,d_{i,+}^*$ or $d_{i,-}^*$. We have:
$$ L(w\sh x)= L(w)\sh (1\otimes x+x\otimes 1),$$
for $w\in \K\langle x, z\rangle$. 
\end{proposition}
\begin{proof}
The formula for $\Delta_\sh$ is due to the fact that $\Delta_\sh$ is a morphism for the shuffle product and that $\Delta_\sh(x)=1\otimes x +x \otimes 1$.  
We now prove the formula for $L=d_{i,+}^*$. We have seen previously that, for $a,b \in \K \langle x,z \rangle$:
\begin{equation}\label{sh}
 azb\sh x=(a \sh x)z b+az(b\sh x). 
\end{equation} 
Hence, $ w\sh x= (w_{i,+,L}\sh x)zw_{i,+}^R+ w_{i,+,L}z(w_{i,+}^R\sh x)$, where $w_{i,+}^L=w_{i,+,L}z$, and by proposition \ref{di*}:
\begin{align*}
d_{i,+}^*(w\sh x)&=(1\otimes  (w_{i,+,L}\sh x)z)\Delta_{\sh}(w_{i,+}^R) +(1\otimes w_{i,+,L}z)\Delta_{\sh}(w_{i,+}^R\sh x)=A+B.
\end{align*}
 \begin{align*}
\text{where} \quad A&=(1\otimes  (w_{i,+,L}\sh x)z)\Delta_{\sh}(w_{i,+}^R) +(1\otimes w_{i,+,L}z)(\Delta_{\sh}(w_{i,+}^R) \sh (1\otimes x))\\
&\overset{(\ref{sh})}{=}((1\otimes w_{i,+}^L)\Delta_\sh(w_{i,+}^R))\sh (1\otimes x)=d_{i,+}^*(w)\sh (1\otimes x)
\end{align*} 
\begin{align*}
\text{and} \quad B&=(1\otimes w_{i,+,L}z)(\Delta_{\sh}(w_{i,+}^R) \sh (x\otimes 1))\\&
=((1\otimes w_{i,+}^L)\Delta_{\sh}(w_{i,+}^R))\sh (x\otimes 1)=d_{i,+}^*(w) \sh (x\otimes 1).
\end{align*}
This proves the proposition for $L=d_{i,+}^*$. We prove the proposition for $L=d_{i,-}^*$ using similar arguments and the decomposition $w\sh x= (x\sh w_{i,-}^L) zw_{i,-,R} + w_{i,-}^Lz(x\sh w_{i,-,R})$, where $w_{i,-}^R=zw_{i,-,R}$.
\end{proof}

\begin{corollary}\label{M*Ihara Kxzx}
The space $\phi(U)=\K\oplus \K \langle x, z \rangle \sh x $ is a coideal for $ \{-,-\}^*$, i.e:
 $$\{-,-\}^*(\phi(U))=\phi(U)\otimes \K \langle x, z \rangle + \K \langle x, z \rangle \otimes\phi(U).$$
\end{corollary}
\begin{proof}
The map $\{-,-\}^*$ is weight graded. Hence, the image of the weight $0$ part $\K$ of $\K \langle x, z \rangle$ by $\{-,-\}^*$ lies in the weight $0$ part $\K^{\otimes2}$ of $\K \langle x, z \rangle^{\otimes2}$. The result follows from this, the previous proposition and the decomposition of the Ihara cobracket (proposition \ref{decomp coIhara}).
\end{proof}
\begin{corollary}\label{coideal U}
$U$ is a coideal for $\phi^\star \{-,-\}^*$.
\end{corollary}

\section{Proof of the main results}\label{main results}
In this section we prove the main results $(a),(b),(c)$ and $(d)$ of the paper announced in the introduction of the paper. In the first subsection we establish in theorem \ref{equiv C} an equivalence between different conditions: $(C1)$ $\ls$ is preserved by the Ihara bracket, $(C2)$ $F$ is a coideal for $\phi^\star\{-,-\}^*$, $(C3)$ $W_R$ is a coideal for $\tilde{\delta}$, $(C4)$ $\D$ is a Lie coalgebra under the formulas given by Goncharov. This proves result $(d)$ stating the equivalence between $(C1)$ and $(C4)$. In the same theorem we show that if one of the previous (equivalent) conditions is satisfied then: (1) $\phi^\star\{-,-\}^*$ induces a Lie cobraket $\delta_{V/F}$ on the quotient $V/F$ for which $\varepsilon: V/F\to \ls^\vee$ of theorem \ref{theorem duality} is an isomorphism of bigraded Lie coalgeras, (2) we have a diagram of bigraded Lie coalgebra isomorphisms $\D\to W/W_R\to V/F \to \ls^\vee$. Using the fact that $(C1)$ was proved by Schneps (or the fact that $(C4)$ was proved by Goncharov) we deduce that the previous conditions and statements are all true. In particular, $\D$ is isomorphic to the bigraded Lie coalgebra dual to $(\ls,\{-,-\})$, proving the main result $(a)$. To prove theorem \ref{equiv C} we use essentially theorem \ref{theorem duality}, proposition \ref{comp Ih delta} and corollary \ref{M*Ihara Kxzx}.\\
In the second subsection we prove results $(c)$ (paragraph \ref{Proof (c)}) and $(d)$ (paragraph \ref{Proof (d)}). To prove $(c)$ we introduce a graded isomorphism $h_m: \K[x_1,\dots,x_m]^\vee \to W_m$ (the subscript $m$ is for the depth $m$ component) and show in proposition \ref{prop hWR}, using corollary \ref{cor ker coeff} of the reminders section, that the image by $h_m$ of the linear forms orthogonal to $\Dsh$ is $(W_R)_m$. This proves that $W_m/(W_R)_m$ is isomorphic to $\Dsh_m^\vee$ and $(c)$ follows since $W/W_R$ is bigraded isomorphic to $\D$ by proposition \ref{definition dihedral}. To prove $(d)$ we consider the graded map $h_m: \K[x_1,\dots,x_m]^\vee \to W_m$, the natural inclusion $i: W_m \to V_m$ and the map $\beta_m:V_m \to \K\langle x,z \rangle_m^\vee$ induced by the pairing $\langle-,-\rangle_\phi$ (the indices $m$ are for depth $m$ components). By theorem \ref{theorem duality}, remark \ref{rmk V/F} and proposition \ref{cor Dsh W} the graded map $\theta=\beta_m\circ i \circ h_m$ induces, with respect to the restriction maps $\K[x_1,\dots,x_m]^\vee \to\Dsh_m^\vee$ and $\K\langle x,z \rangle_m^\vee \to \ls_m^\vee$, an isomorphism $\Dsh_m^\vee\to \ls_m^\vee$. We show using a direct computation that $\theta=f_m^\vee$ and then deduce $(d)$ using an algebraic argument.

\subsection{Proof of main results (a) and (b)}\label{proof a,b}
Denote by $\{-,-\}_{\vert \ls}^{\bar{\vee}}$ the bigraded dual map to the Ihara bracket $\{-,-\}$ restricted to $\ls$. The bigraded dual $\ls^\vee$ of $\ls$ equipped with $\{-,-\}_{\vert \ls}^{\bar{\vee}}$ is a bigraded Lie coalgebra (see section \ref{Section rappel}). Recall that $\{-,-\}^*$ is the adjoint of $\{-,-\}$ with respect to $\langle-,-\rangle$ and $\langle-,-\rangle^{\otimes 2}$ and that $\phi$ is the isomorphism of proposition \ref{morphisme phi}.
\begin{theorem}\label{equiv C}
\baselineskip 18pt
\begin{itemize}
\item[1)] The following conditions are all equivalent: 
 \begin{itemize}
\item[(C1)] $\ls$ is preserved by the Ihara bracket $\{-,-\}$.
\item[(C2)] $F$ is a coideal for $\phi^\star\{-,-\}^*: V \to V\otimes V$, i.e. $\phi^\star\{-,-\}^*(F)\subset F\otimes V+V\otimes F$. 
\item[(C3)] $W_R$ is a coideal for $\tilde{\delta}:W\to W\otimes W$.
\item[(C4)] $\D$ is a Lie coalgebra under the formulas given by Goncharov. 
\end{itemize}
\item[2)] Assume that one of the equivalent conditions in $(1)$ holds:
\begin{itemize}
 \item[2.1)]$\phi^\star \{-,-\}^*$ induces a bigraded Lie cobracket $\delta_{V/F}:V/F \to V/F \otimes V/F$ such that $(V/F,\delta_{V/F})\overset{\varepsilon}{\longrightarrow} (\ls^\vee,\{-,-\}_{\vert \ls}^{\bar{\vee}})$, with $\varepsilon$ the map of theorem \ref{theorem duality}, is an isomorphism of bigraded Lie coalgebras.
\item[2.2)] We have the following diagram of bigraded Lie coalgebra isomorphisms:
           $$ \D \overset{\eta}{\longrightarrow} (W/W_R,\delta) \overset{\bar{i}}{\longrightarrow}(V/F,\delta_{V/F}) \overset{\varepsilon}{\longrightarrow}  (\ls^\vee,\{-,-\}_{\vert \ls}^{\bar{\vee}}),$$
where $\eta$ is the isomorphism of proposition \ref{definition dihedral}, $\bar{i}$ is the isomorphism induced by the natural inclusion $W\subset V=W\oplus U$  and $\varepsilon$ is as in $(2.1)$.
\end{itemize}
\end{itemize}
\end{theorem}
\begin{proof}
 We have seen in theorem \ref{theorem duality} that $F$ is the space orthogonal to $\ls$ with respect to the perfect bigraded pairing $\langle -,- \rangle_\phi=\langle \phi(-),- \rangle: V \otimes\K\langle x,z\rangle  \to \K$. Moreover, the adjoint of $\{-,-\}$ with respect to $\langle -,- \rangle_\phi$ is $\phi^\star\{-,-\}^*$. Tese facts imply the equivalence $(C1) \Leftrightarrow (C2)$ and point $(2.1)$ of this proposition (see proposition \ref{Lie pairing}). By definition $V=W\oplus U$ and $F=W_R \oplus U$. Therefore, $(W\otimes W) (\cap F\otimes V+V\otimes F)= W_R\otimes W +W\otimes W_R$, and the equivalence $(C2) \Leftrightarrow (C3)$ follows from the facts: $U$ is a coideal for $\phi^*\{-,-\}^*$ (corollary \ref{coideal U}) and $\tilde{\delta}-\phi^\star\{-,-\}^*(W_R)\subset F\otimes V + V\otimes F$ (proposition \ref{comp Ih delta}). The equivalence $(C3) \Leftrightarrow (C4)$ was established previously (remark \ref{rmk DW}). We still have to prove $(2.2)$. We have proven that $\eta$ is an isomorphism of bigraded Lie coalgebras in theorem \ref{Lie structure dihedral}. Since $(\tilde{\delta}-\phi^\star\{-,-\}^*(W_R))\subset F\otimes V + V\otimes F$ as mentioned earlier in this proof, the isomorphism $\bar{i}$ respects the Lie cobracket. The isomorphism $\varepsilon$ is Lie by $(2.1)$. We have proved the theorem.
\end{proof}
\begin{corollary}\label{cor th}
The conditions $(C1),\dots,(C4)$ are satisfied and the statements $(2.1)$ and $(2.2)$ are true.
\end{corollary}
\begin{proof}
We have already seen that $(C1)$ is proved in \cite{Schnepsari} and $(C4)$ is proved in \cite{Gon}.
\end{proof}
We have proved the main results $(a)$ and $(b)$ announced in the introduction of the paper. Indeed, $(a)$ states that the dihedral Lie coalgebras $\D$ and $(\ls^\vee,\{-,-\}_{\vert \ls}^{\bar{\vee}})$ are isomorphic bigraded Lie coalgebra and this follows form $(2.2)$, and $(b)$ corresponds to the equivalence $(C1)\Leftrightarrow (C4)$.

\subsection{Proof of main results (c) and (d)}\label{Subsection Dsh ls D}

\subsubsection{\textbf{Proof of $(c)$:}}\label{Proof (c)}
Denote by $W_m$ and $(W_R)_m$ the depth $m$ components of the spaces $W$ and $W_R$ respectively. The depth $m$ components are weight graded. We equip $\K[x_1,\dots,x_m]$ with the grading $\oplus_{k\geq m} K_{k}$, where $K_k$ is the space of polynomials of degree $m-k$ and define for $m\geq 2$ the graded isomorphism $h_m: \K[x_1,\dots,x_m]^\vee \to W_m$, by: 
$$ h_m(\varphi)= \sum_{n_i\geq 1}\varphi(x_1^{n_m-1}\cdots x_m^{n_1-1}) I(n_1,\dots,n_m),$$
for $\varphi  \in \K[x_1,\dots,x_m]^\vee$. The above sum make sense since only finite number of terms is non zero. 
\begin{proposition}\label{prop hWR}
Take $m\geq 2$, $h_m$ as in the previous paragraph and denote by $\Dsh_m'$ the subspace of $\K[x_1,\dots,x_m]^\vee$ of forms orthogonal to $\Dsh_m$. We have $h_m(\Dsh_m')=(W_R)_m$ 
\end{proposition}
\begin{proof}
Set $H= \sum_{ni\geq 1}  I(n_1,\dots,n_m)\otimes x_1^{n_m-1}\cdots x_m^{n_1-1}$. We recall that $$\Dsh_m= \cap_{l=1}^{m-1}( \mathrm{Ker}(T_{m,*}^{(l)}) \cap \mathrm{Ker}(T_{m,\sh}^{(l)}))),$$
where $T_{m,*}^{(l)}$ and $T_{m,\sh}^{(l)}$ are the graded endomorphisms of $\K[x_1,\dots,x_m]$ defined in subsection \ref{def Dshm}. Applying corollary \ref{cor ker coeff} for $A=W_m, B=\K[x_1,\dots,x_m], S=H$ and $\{T_1,\dots,T_k\}=\{T_{m,*}^{(l)},T_{m,\sh}^{(l)}\}_{l\in[1,m-1]}$, we get that: 
$$h_m(\Dsh_m')=\sum_{l=1}^{m-1} (W_m((\mathrm{id}\otimes  T_{m,*}^{(l)})(H))+W_m((\mathrm{id}\otimes  T_{m,\sh}^{(l)})(H)),$$
where $W_m(Z) \subset W_m$ is the space generated by the coefficients of the series $Z$. Notice that for $m\geq 2$:
$$(W_R)_m=\sum_{l=1}^{m-1} (W_m((\mathrm{id}\otimes  T_{m,*}^{(l)}R_m)(H))+W_m((\mathrm{id}\otimes  T_{m,*}^{(l)}L_mR_m)(H)),$$
where $R_m$ and $L_m$  are the algebra automorphisms of $\K[x_1, \dots , x_m]$ given by $R_m(x_i)= x_{m+1-i}$ and $L_m(x_i)=x_1+\cdots+x_i$ for $i\in[1,m]$. One checks that: 
$$ R_m T_{m,*}^{(l)}R_m= T_{m,*}^{(m-l)} \quad \text{and} R_m T_{m,*}^{(l)}L_mR_m=T_{m-l,\sh}^{(l)}.$$
Since $R_m$ is an involution and in particular an automorphism, we get:
\begin{align*}(W_R)_m &=\sum_{l=1}^{m-1} (W_m((\mathrm{id}\otimes  R_mT_{m,*}^{(l)})(H))+W_m((\mathrm{id}\otimes R_m T_{m,\sh}^{(l)})(H)),\\
&=\sum_{l=1}^{m-1} (W_m((\mathrm{id}\otimes T_{m,*}^{(l)})(H))+W_m((\mathrm{id}\otimes  T_{m,\sh}^{(l)})(H)).
\end{align*}
The last decomposition of $(W_R)_m$ correspond to the one obtained for $h_m(\Dsh_m')$. We have proved the proposition.
\end{proof}
Let $g_m: \K[x_1,\dots,x_m]^\vee/ \Dsh'_m\to \Dsh_m^\vee$ and $\bar{h}_m:  \K[x_1,\dots,x_m]^\vee/\Dsh'_m \to W_m/(W_R)_m$ be the isomorphisms induced by the restriction morphism $\K[x_1,\dots,x_m]^\vee\to  \Dsh_m^\vee$ and the isomorphism $h_m$ respectively.
\begin{corollary}\label{cor Dsh W}
The maps $\bar{h}_mg_{m}^{-1}: \Dsh_m^\vee \to W_m/(W_R)_m$ and $\eta^{-1}\bar{h}_mg_{m}^{-1}:\Dsh_m^\vee\to D_{m,\bullet}$, where $\eta$ is the isomorphism of proposition \ref{definition dihedral} and $D_{m,\bullet}$ is the depth $m$ component of $\D$, are graded isomorphisms.  
\end{corollary}
We have proved the main result $(c)$. Note that the above corollary with corollary \ref{cor th} imply that $\ls_m$ (subscript fo depth $m$ component) and $\Dsh_m^\vee$ are isomorphic graded spaces.
\subsubsection{\textbf{Proof of $(d)$:}}\label{Proof (d)}
Denote by $V_m,F_m, \ls_m$ and $\K \langle x,z \rangle_m$ the depth $m$ components of $V, F, \ls$ and $\K \langle x,z \rangle$ respectively and define $\beta_m: V_m \to \K \langle x,z \rangle_m^\vee$ by $v\mapsto \langle \phi(v),-\rangle$. By theorem \ref{theorem duality}, remark \ref{rmk V/F} and proposition \ref{cor Dsh W}, we have the following commutative diagram:
\[\begin{tikzcd}
\K[x_1,\dots,x_m]^\vee \arrow{r}{h_m} \arrow{d}[swap]{r_1} & W_m \arrow{r}{i} \arrow{d}{ } & V_m \arrow{d} \arrow{r}{\beta_m}&  \K \langle x,z \rangle_m^\vee \arrow{d}{r_2} \\
\Dsh^\vee \arrow{r}[swap]{\bar{h}_mg_m^{-1}}{\simeq}& W_m/(W_R)_m\arrow{r}[swap]{\bar{i}}{\simeq}& V_m/F_m\arrow{r}[swap]{\varepsilon}{\simeq}& \ls_m^\vee
\end{tikzcd},\quad (A)\]
where $r_1$ and $r_2$ are the restriction morphisms and the two other vertical morphisms are the quotient maps. For $m\geq 2$, we define the linear map $f_m :  \K \langle x,z \rangle_m \to \K[x_1,\dots,x_m]$  by:
$$ f_m(x^{n_1}z\cdots x^{n_m}zx^{n_{m+1}})=\delta_{0n_{m+1}}x_1^{n_1}x_2^{n_2}\cdots x_m^{n_m},$$
with $\delta_{ab}$ the Kronecker delta. 
\begin{lemma}\label{final lemma} 
For $m\geq 2$, the composition of the top morphisms $\beta_m, i$ and $h_m$ of diagram $(A)$ is equal to $f_m^\vee$.
\end{lemma}
\begin{proof}
Take $w=x^{l_1}z\cdots x^{l_m}zx^{l_{m+1}}$. We have for $\varphi\in\K[x_1,\dots,x_m]^\vee$:
\begin{align*}
\beta_m ih_m(\varphi)(w)&=\langle \phi(\sum_{n_i\geq 1} \varphi(x_1^{n_m-1}\cdots x_m^{n_1-1}))I(n_1,\dots,n_m), w\rangle\\
&=\sum_{n_i\geq 1} \varphi(x_1^{n_m-1}\cdots x_m^{n_1-1})\langle x^{n_m-1}z\cdots x^{n_1-1}z, x^{l_1}z\cdots x^{l_m}zx^{l_{m+1}}\rangle\\
&=\delta_{0l_{m+1}}\varphi(x_1^{l_1}\cdots x_m^{l_m})=\varphi(f_m(w))=f_m^\vee(\varphi)(w).
\end{align*}
This proves the lemma.
\end{proof}

\begin{proposition}\label{compatible morphisms}
For $m \geq 2$, $f_m$ restricts to a graded isomorphism $\bar{f}_m:\ls_m \to \Dsh_m$ and the following diagram of weight graded isomorphisms commutes: 
\[\begin{tikzcd}
 & W_m/(W_R)_m  \arrow{dr}{\varepsilon\circ \bar{i}} \\
 \Dsh_m^\vee \arrow{ur}{\bar{h}_m\circ g_m^{-1}} \arrow{rr}{\bar{f}_m^\vee} && \ls_m^\vee
\end{tikzcd}\]
\end{proposition}
\begin{proof}
It follows from diagram $(A)$ and the previous lemma (lemma \ref{final lemma}) that the following diagram commutes:
\[\begin{tikzcd}
\K[x_1,\dots,x_m]^\vee \arrow{r}{f_m^\vee}\arrow{d}[swap]{r_1}&  \K \langle x,z \rangle_m^\vee \arrow{d}{r_2} \\
\Dsh^\vee \arrow{r}[swap]{\varepsilon\circ \bar{i} \circ \bar{h}_m\circ g_m^{-1}}{\simeq}& \ls_m^\vee
\end{tikzcd},\quad (B)\]
with $r_1$ and $r_2$ the restriction morphisms.
Taking the dual of this diagram we get the following commutative diagram:
\[\begin{tikzcd}
\K[x_1,\dots,x_m]& \arrow{l}[swap]{f_m} \K \langle x,z \rangle_m\\
\Dsh\arrow{u}{i_1} &\arrow{l}[swap]{\simeq} \ls_m \arrow{u}[swap]{i_2}
\end{tikzcd},\]
where $i_1$ and $i_2$ are the natural inclusions. This shows that $f_m$ restricts to a graded  isomorphism $\bar{f}_m:\ls_m \to \Dsh_m$ and that the isomorphism $(\varepsilon\circ \bar{i})\circ (\bar{h}_m\circ g_m^{-1}) :\Dsh^\vee \to \ls_m^\vee$ in diagram $(B)$ is $\bar{f}_m^\vee$. We have proved the proposition.
\end{proof}
We have proved $(d)$. Indeed, one can replace $W_m/(W_R)_m$ by $D_{m,\bullet}$ in the diagram of the last proposition.

\renewcommand{\abstractname}{\textbf{Acknowledgments}}
\begin{abstract} I would like to thank Benjamin Enriquez. A big part of this work was done under his supervision during my master 2 internship (thesis).
\end{abstract}

\bibliographystyle{alpha}
\bibliography{Biblio}{}

 \end{document}